\documentclass[a4paper,english,american]{amsart}
\usepackage{mathptmx}
\usepackage[T1]{fontenc}
\usepackage[latin9]{inputenc}
\usepackage{babel}
\usepackage{refstyle}
\usepackage{units}
\usepackage{amsmath}
\usepackage{amsthm}
\usepackage{amssymb}
\usepackage[unicode=true,pdfusetitle,
 bookmarks=true,bookmarksnumbered=false,bookmarksopen=false,
 breaklinks=false,pdfborder={0 0 1},backref=false,colorlinks=false]
 {hyperref}
\hypersetup{
 bookmarksopen}

\makeatletter

\AtBeginDocument{\providecommand\thmref[1]{\ref{thm:#1}}}
\AtBeginDocument{\providecommand\secref[1]{\ref{sec:#1}}}
\AtBeginDocument{\providecommand\lemref[1]{\ref{lem:#1}}}
\AtBeginDocument{\providecommand\eqref[1]{\ref{eq:#1}}}
\AtBeginDocument{\providecommand\propref[1]{\ref{prop:#1}}}
\pdfpageheight\paperheight
\pdfpagewidth\paperwidth

\RS@ifundefined{subref}
  {\def\RSsubtxt{section~}\newref{sub}{name = \RSsubtxt}}
  {}
\RS@ifundefined{thmref}
  {\def\RSthmtxt{theorem~}\newref{thm}{name = \RSthmtxt}}
  {}
\RS@ifundefined{lemref}
  {\def\RSlemtxt{lemma~}\newref{lem}{name = \RSlemtxt}}
  {}

\theoremstyle{plain}
\numberwithin{equation}{section}
\numberwithin{figure}{section}
\numberwithin{table}{section}
\newtheorem{stthm}{\protect\theoremname}[section]

\newtheorem{spprop}{\protect\propositionname}[section]

 \newref{sec}{%
 name = \RSsectxt,
 names = \RSsecstxt,
 lsttxt = \RSlsttxt,
 lsttwotxt = \RSlsttwotxt,
 refcmd = \ref{#1}}
 \newref{thm}{%
 name = \theoremname~,
 lsttxt = \RSlsttxt,
 lsttwotxt = \RSlsttwotxt,
 refcmd = {\ref{#1}}}
 \usepackage{enumitem}
 \usepackage{enumitem}
 \theoremstyle{remark}
 \newtheorem*{rem*}{\protect\remarkname}
 \newref{rem}{%
 name = \remarkname~,
 lsttxt = \RSlsttxt,
 lsttwotxt = \RSlsttwotxt,
 refcmd = {\ref{#1}}}
 \newref{prop}{%
 name = \propositionname~,
 lsttxt = \RSlsttxt,
 lsttwotxt = \RSlsttwotxt,
 refcmd = {\ref{#1}}}
 \theoremstyle{plain}
 \newtheorem*{thm*}{\protect\theoremname}
 \newref{thm}{%
 name = \theoremname~,
 lsttxt = \RSlsttxt,
 lsttwotxt = \RSlsttwotxt,
 refcmd = {\ref{#1}}}
 \theoremstyle{plain}
 \newtheorem*{cor*}{\protect\corollaryname}
 \newref{cor}{%
 name = \corollaryname~,
 lsttxt = \RSlsttxt,
 lsttwotxt = \RSlsttwotxt,
 refcmd = {\ref{#1}}}
 \theoremstyle{plain}
 \newtheorem{sllem}{\protect\lemmaname}[section]
 \newref{lem}{%
 name = \lemmaname~,
 lsttxt = \RSlsttxt,
 lsttwotxt = \RSlsttwotxt,
 refcmd = {\ref{#1}}}
 \def\SEI#1{\setcounter{enumi}{#1}}
 \newref{sec}{%
 name = \RSsectxt,
 names = \RSsecstxt,
 lsttxt = \RSlsttxt,
 lsttwotxt = \RSlsttwotxt,
 refcmd = \ref{#1}}

\newcommand\mynobreakpar{\par\nobreak\@afterheading}
\usepackage{bbm, yhmath}

\makeatother

 \addto\captionsamerican{\renewcommand{\corollaryname}{Corollary}}
 \addto\captionsamerican{\renewcommand{\lemmaname}{Lemma}}
 \addto\captionsamerican{\renewcommand{\propositionname}{Proposition}}
 \addto\captionsamerican{\renewcommand{\remarkname}{Remark}}
 \addto\captionsamerican{\renewcommand{\theoremname}{Theorem}}
 \addto\captionsenglish{\renewcommand{\corollaryname}{Corollary}}
 \addto\captionsenglish{\renewcommand{\lemmaname}{Lemma}}
 \addto\captionsenglish{\renewcommand{\propositionname}{Proposition}}
 \addto\captionsenglish{\renewcommand{\remarkname}{Remark}}
 \addto\captionsenglish{\renewcommand{\theoremname}{Theorem}}
 \providecommand{\corollaryname}{Corollary}
 \providecommand{\lemmaname}{Lemma}
 \providecommand{\propositionname}{Proposition}
 \providecommand{\remarkname}{Remark}
 \providecommand{\theoremname}{Theorem}

\begin{document}
\selectlanguage{english}%
\def\rightmark{WEIGHTED $L^p$ ESTIMATES FOR SOLUTIONS OF THE $\bar\partial$ EQUATION AND APPLICATIONS}
\def\leftmark{P. CHARPENTIER \& Y. DUPAIN}
\def\RSsectxt{Section~}%

\selectlanguage{american}%

\title{Weighted and Boundary $L^{p}$ Estimates for Solutions of the $\overline{\partial}$-equation
on Lineally Convex Domains of Finite Type and Applications}

\author{P. Charpentier \& Y. Dupain}
\begin{abstract}
We obtain sharp weighted estimates for solutions of
the equation $\overline{\partial}u=f$ in a lineally convex domain
of finite type. Precisely we obtain estimates in the spaces $L^{p}(\Omega,\delta^{\gamma})$,
$\delta$ being the distance to the boundary, with two different types
of hypothesis on the form $f$: first, if the data $f$ belongs to
$L^{p}\left(\Omega,\delta_{\Omega}^{\gamma}\right)$, $\gamma>-1$,
we have a mixed gain on the index $p$ and the exponent $\gamma$;
secondly we obtain a similar estimate when the data $f$ satisfies
an apropriate anisotropic $L^{p}$ estimate with weight $\delta_{\Omega}^{\gamma+1}$.
Moreover we extend those results to $\gamma=-1$ and obtain $L^{p}(\partial\Omega)$
and $BMO(\partial\Omega)$ estimates. These results allow us to extend
the $L^{p}(\Omega,\delta^{\gamma})$-regularity results for weighted
Bergman projection obtained in \cite{CDM} for convex domains to more
general weights.
\end{abstract}



\address{P. Charpentier, Universit\'e Bordeaux I, Institut de Math\'ematiques
de Bordeaux, 351, Cours de la Lib\'eration, 33405, Talence, France}

\email{P. Charpentier: philippe.charpentier@math.u-bordeaux.fr}

\maketitle

\section{Introduction}

Sharp estimates for solutions of the $\overline{\partial}$-equation
are a fundamental tool to study various problems in complex analysis
of several variables.

In this paper we consider the case of smoothly bounded lineally convex
domain of finite type $\Omega$ in $\mathbb{C}^{n}$. Precisely, we
obtain new weighted $L^{p}$ estimates for solutions of the equation
$\overline{\partial}u=f$, $f$ being a $\left(0,r\right)$-form,
$1\leq r\leq n-1$, in $\Omega$.

Our first result is a weighted $L^{p}\left(\Omega,\delta_{\Omega}^{\gamma}\right)$
estimate ($\delta_{\Omega}$ being the distance to the boundary of
$\Omega$) with a mixed gain on $p$ and $\gamma$ extending results
obtained (without weights) for convex domains of finite type by various
authors (for example, A. Cumenge in \cite[Theorem 1.2]{Cumenge-estimates-holder}
and B. Fisher in \cite[Theorem 1.1]{MR1815835}) (see also T. Hefer
\cite{Hef02}).

The second one gives a weighted $L^{p}\left(\Omega,\delta_{\Omega}^{\gamma}\right)$,
$\gamma>-1$, estimate with gain on $\gamma$, with a non isotropic
hypothesis $\left\Vert f\right\Vert _{\mathbbmss k,p}$ (formula (\ref{eq:punctual-norm-k-p}))
on the form $f$ generalizing the estimate obtained by A. Cumenge
(\cite[Theorem 1.3]{Cumenge-Navanlinna-convex}) in convex domains
of finite type for $p=1$ with the ``norm'' $\left\Vert f\right\Vert _{\mathbbmss k}$
defined in \cite{Bruna-Charp-Dupain-Annals}. As far as we know, the
estimate presented here (for $p>1$) was only stated for strictly
pseudoconvex domains in \cite[Theorem 1.4 and Remark that follows]{Cha80}.

The two last results are $L^{p}(\partial\Omega)$ estimates. The first
is the limit case of the previous one vhen $\gamma$ tends to $-1$
and the second one is a $BMO(\partial\Omega)$ and $L^{p}(\partial\Omega)$
estimate with an hypothesis based on Carleson measure for $\left(0,1\right)$-forms
only.

To prove these results, we use the method introduced in \cite{CDMb},
which overcomes the fact that the Diederich-Fornaess support function
is only locally defined and that it is not possible to extend it to
the whole domain (like W. Alexandre did it in the convex case in \cite{Ale01})
using a division with good estimates, our domain being non convex.

\medskip{}

These estimates are used next to generalize an estimate for weighted
Bergman projections obtained in \cite{CDM} for convex domains of
finite type.

The study of the regularity of the Bergman projection onto holomorphic
functions in a given Hilbert space is a very classical subject. When
the Hilbert space is the standard Lebesgue $L^{2}$ space on a smoothly
bounded pseudoconvex domain $\Omega$ in $\mathbb{C}^{n}$, many results
are known and there is a very large bibliography.

When the Hilbert space is a weighted $L^{2}$ space on a smoothly
bounded pseudoconvex domain $\Omega$ in $\mathbb{C}^{n}$, it is
well known for a long time that the regularity of the Bergman projection
depends strongly on the weight (\cite{Kohn-defining-function}, \cite{Bar92},
\cite{Christ96}). Until last years few results where known (see \cite{FR75},
\cite{Lig89}, \cite{BG95}, \cite{CDC97}) but recently some positive
and negative results where obtained by several authors (see for example
\cite{Zey11}, \cite{Zey12}, \cite{Zey13a}, \cite{Zey13b}, \cite{CDM},
\cite{CPDY}, \cite{CZ}, \cite{Zey} and references therein).

Let $\Omega$ be a convex domain of finite type $m$ in $\mathbb{C}^{n}$.
Let $g$ be a gauge function for $\Omega$ and define $\rho_{0}=g^{4}e^{1-\nicefrac{1}{g}}-1$.
Let $P_{\omega_{0}}$ be the Bergman projection of the space $L^{2}\left(\Omega,\omega_{0}\right)$,
where $\omega_{0}=\left(-\rho_{0}\right)^{r}$, $r\in\mathbb{Q}_{+}$.
Then in \cite[Theorem 2.1]{CDM} we proved that $P_{\omega_{0}}$
maps continuously the spaces $L^{p}\left(\Omega,\delta_{\Omega}^{\beta}\right)$,
$p\in\left]1,+\infty\right[$, $0<\beta+1\leq p(r+1)$, into themselves.
Here we consider a weight $\omega$ which is a non negative rational
power of a $\mathcal{C}^{2}$ function in $\overline{\Omega}$ equivalent
to the distance to the boundary and we prove that the Bergman projection
$P_{\omega}$ of the Hilbert space $L^{2}\left(\Omega,\omega\right)$
maps continuously the spaces $L^{p}\left(\Omega,\delta_{\partial\Omega}^{\beta}\right)$,
$p\in\left]1,+\infty\right[$, $0<\beta+1\leq r+1$ into themselves
and the lipschitz spaces $\Lambda_{\alpha}(\Omega)$, $0<\alpha\leq\nicefrac{1}{m}$,
into themselves.

This result is obtained comparing the operators $P_{\omega_{0}}$
and $P_{\omega}$ with the method described in \cite{CPDY}. To do
it, we use the weighted $L^{p}\left(\Omega,\delta_{\Omega}^{\gamma}\right)$
estimates with appropriate gains on the index $p$ and on the power
$\gamma$ for solution of the $\overline{\partial}$-equation obtained
in the first part.

\section{Notations and main results}

Throughout this paper we will use the following general notations:
\begin{itemize}
\item $\Omega$ is a smoothly bounded lineally convex domain of finite type
$m$ in $\mathbb{C}^{n}$. Precisely (c.f. \cite{CDMb}) ``lineally
convex'' means that, for all point in the boundary $\partial\Omega$
of $\Omega$, there exists a neighborhood $W$ of $p$ such that,
for all point $z\in\partial\Omega\cap W$,
\[
\left(z+T_{z}^{1,0}\right)\cap(D\cap W)=\emptyset,
\]
where $T_{z}^{1,0}$ is the holomorphic tangent space to $\partial\Omega$
at the point $z$. Furthermore, we can assume that there exists a
a smooth defining function $\rho$ of $\Omega$ such that, for $\delta_{0}$
sufficiently small, the domains $\Omega_{t}=\left\{ \rho(z)<t\right\} $,
$-\delta_{0}\leq t\leq\delta_{0}$, are all lineally convex of finite
type $\leq m$.
\item $\delta_{\Omega}$ denotes the distance to the boundary of $\Omega$.
\item For any real number $\gamma>-1$, we denote by $L^{p}\left(\Omega,\delta_{\Omega}^{\gamma}\right)$
the $L^{p}$-space on $\Omega$ for the measure $\delta_{\Omega}^{\gamma}(z)d\lambda(z)$,
$\lambda$ being the Lebesgue measure.
\item $BMO(\Omega)$ denotes the standard $BMO$ space on $\Omega$ and
$\Lambda_{\alpha}(\Omega)$ the standard lipschitz space.
\end{itemize}

Our first results give sharp $L^{q}\left(\Omega,\delta_{\Omega}^{\gamma'}\right)$,
$BMO(\Omega)$ and $\Lambda_{\alpha}(\Omega)$ estimates for solutions
of the $\overline{\partial}$-equation in $\Omega$ with data in $L^{p}\left(\Omega,\delta_{\Omega}^{\gamma}\right)$:
\begin{stthm}
\label{thm:d-bar-q-gamma'-p-gamma-lip}Let $N$ be a positive large
integer. let $\gamma$ and $\gamma'$ be two real numbers such that
$\gamma'>-1$ and $\gamma-\nicefrac{1}{m}\leq\gamma'\leq\gamma\leq N-2$.
Then there exists a linear operator $T$, depending on $\rho$ and
$N$, such that, for any $\overline{\partial}$-closed $\left(0,r\right)$-form
($1\leq r\leq n-1$) with coefficients in $L^{p}\left(\Omega,\delta_{\Omega}^{\gamma}\right)$,
$p\in\left[1,+\infty\right]$, $Tf$ is a solution of the equation
$\overline{\partial}(Tf)=f$ such that:
\begin{enumerate}
\item If $1\leq p<\frac{m(\gamma'+n)+2-(m-2)(r-1)}{1-m(\gamma-\gamma')}$,
then $T$ maps continuously the space of $\overline{\partial}$-closed
forms with coefficients in $L^{p}\left(\Omega,\delta_{\Omega}^{\gamma}\right)$
into the space of forms whose coefficients are in $L^{q}\left(\Omega,\delta_{\Omega}^{\gamma'}\right)$
with $\frac{1}{q}=\frac{1}{p}-\frac{1-m(\gamma-\gamma')}{m(\gamma'+n)+2-(m-2)(r-1)}$;
\item If $p=m(\gamma+n)+2-(m-2)(r-1)$, then $T$ maps continuously the
space of $\overline{\partial}$-closed forms with coefficients in
$L^{p}\left(\Omega,\delta_{\Omega}^{\gamma}\right)$ into the space
of forms whose coefficients are in $BMO(\Omega)$;
\item If $p\in\left]m(\gamma+n)+2-(m-2)(r-1),+\infty\right]$, then $T$
maps continuously the space of $\overline{\partial}$-closed forms
with coefficients in $L^{p}\left(\Omega,\delta_{\Omega}^{\gamma}\right)$
into the space of forms whose coefficients are in the lipschitz space
$\Lambda_{\alpha}(\Omega)$ with $\alpha=\frac{1}{m}\left[1-\frac{m(\gamma+n)+2-(m-2)(r-1)}{p}\right]$.
\end{enumerate}
\end{stthm}

\begin{rem*}
\quad\mynobreakpar
\begin{itemize}
\item Note that, if $\gamma'<\gamma$, then $\frac{m(\gamma'+n)+2-(m-2)(r-1)}{1-m(\gamma-\gamma')}>m(\gamma+n)+2-(m-2)(r-1)$,
and (3) is sharper than (1) for 
\[
p\in\left]m(\gamma+n)+2-(m-2)(r-1),\frac{m(\gamma'+n)+2-(m-2)(r-1)}{1-m(\gamma-\gamma')}\right[.
\]

\item For $\gamma=\gamma'=0$ and $r=1$, these estimates are known to be
sharp (see \cite{CKM93}).
\item For $m>2$ and $r\geq2$, the above result is strictly better than
the one obtain for convex domains by A Cumenge in \cite[Theorem 1.2]{Cumenge-estimates-holder}
and B. Fisher in \cite[Theorem 1.1]{MR1815835} and for complex ellipsoids
in \cite{CKM93}.
\end{itemize}
\end{rem*}

\bigskip{}

The two next propositions, which are immediate corollaries of the
theorem, will be used in the last section:
\begin{spprop}
\label{prop:d-bar-gain-weight}For all large integer $N$, there exists a linear operator
$T$ solving the $\overline{\partial}$-equation in $\Omega$ such
that, for all $p\in\left[1,+\infty\right[$ and all $\gamma$, $-1<\gamma\leq N-3$, there exists a constant
$C_{N,p,\gamma}>0$ such that, for all $\overline{\partial}$-closed $\left(0,r\right)$-form $f$, $1\leq r\leq n-1$, we have
\[
\int_{\Omega}\left|Tf\right|^{p}\delta_{\Omega}^{\gamma}d\lambda\leq C_{N,p,\gamma}\int_{\Omega}\left|f\right|^{p}\delta_{\Omega}^{\gamma+\nicefrac{1}{m}}d\lambda.
\]

\end{spprop}

\begin{spprop}
\label{prop:d-bar-gain-exponent}For all large integer $N$, there exists a constant $\varepsilon_{N}>0$
and a linear operator $T$ solving the $\overline{\partial}$-equation in $\Omega$ such
that, for and all $p\in\left[1,+\infty\right[$ and all $\gamma$, $-1<\gamma\leq N-3$, there exists a constant
$C_{N,p,\gamma}>0$ such that for all $\overline{\partial}$-closed $\left(0,r\right)$-form $f$, $1\leq r\leq n-1$, we have
\[
\int_{\Omega}\left|Tf\right|^{p+\varepsilon_{N}}\delta_{\Omega}^{\gamma}d\lambda\leq C_{N,p,\gamma}\int_{\Omega}\left|f\right|^{p}\delta_{\Omega}^{\gamma}d\lambda.
\]

\end{spprop}

\bigskip{}

In \cite{Cumenge-Navanlinna-convex} A. Cumenge obtained a weighted
$L^{1}$-anisotropic estimate for solutions of the $\overline{\partial}$-equation
for convex domains of finite type, using the punctual norm $\left\Vert f\right\Vert _{\mathbbmss k}$
introduced in \cite{Bruna-Charp-Dupain-Annals}:
\begin{thm*}[A. cumenge \cite{Cumenge-Navanlinna-convex}]
Let $D$ be a convex domain of finite type. There exists a constant
$C>0$ such that, for all $\alpha>0$ and all smooth $\overline{\partial}$-closed
$\left(0,1\right)$-form $f$ on $\overline{D}$, there exists a solution
of the equation $\overline{\partial}u=f$, continuous on $\overline{D}$
such that
\[
\int_{D}\left|u\right|\delta_{D}^{\alpha-1}d\lambda\leq C\frac{1}{\min\left\{ \alpha,1\right\} }\int_{D}\left\Vert f\right\Vert _{\mathbbmss k}\delta_{D}^{\alpha}d\lambda.
\]

\end{thm*}

The estimate given by \thmref{d-bar-q-gamma'-p-gamma-lip} when $p=q=1$
(and then $\gamma'=\gamma-\nicefrac{1}{m}$) is weaker than the one
given above. Then it is natural to ask if the above estimate can be
extended to weighted $L^{p}$ norms. For example, if $D$ is a smooth
strictly pseudoconvex domain it is proved in \cite[Th\'eor\`eme 1.4]{Cha80}
that: \emph{for $\alpha>0$, if $f$ is a smooth $\overline{\partial}$-closed
form such that the coefficients of $f$ are in $L^{p}\left(D,\delta_{D}^{\alpha}\right)$
and the coefficients of $f\wedge\overline{\partial}\rho$ are in $L^{p}\left(D,\delta_{\Omega}^{\alpha-\nicefrac{1}{2}}\right)$,
$1\leq p<+\infty$, $\alpha>0$, then there exists a solution of the
equation $\overline{\partial}u=f$ which is in $L^{p}\left(D,\delta_{D}^{\alpha-1}\right)$}.

Our next result extends A. Cumenge's theorem and \cite{Cha80} theorem
to $\left(0,r\right)$-forms in lineally convex domains of finite
type for $1\leq p<+\infty$.

To state it, we need to extend the definition of the punctual anisotropic
norm $\left\Vert .\right\Vert _{\mathbbmss k}$ given in \cite{CDMb}
to the $L^{p}$ context.

We first introduce new quantities associated to the geometry: for
$z$ close to the boundary of $\Omega$, let us define
\[
\sigma_{1}(z)=\delta_{\Omega}(z),
\]
and, for $2\leq k\leq n$
\[
\sigma_{k}(z)=\frac{1}{\prod_{j=1}^{k-1}\sigma_{j}(z)}\int_{B_{\delta_{\Omega}(z)}(z)}\frac{d\lambda(\zeta)}{\left|\zeta-z\right|^{2(n-k)+1}},
\]
where $B_{\delta_{\Omega}(z)}(z)$ is the pseudo-ball defined by (\ref{eq:def-pseudo-balls}).

let $f$ be a $\left(0,r\right)$-form whose coefficients are functions; for $z$ close to the boundary
and $p\in\left[1,+\infty\right[$, we define
\begin{equation}
\left\Vert f(z)\right\Vert _{\mathbbmss k,p}=\sup_{\left\Vert v_{i}\right\Vert =1}\left|\left\langle f;\overline{v_{1}},\ldots,\overline{v_{n}}\right\rangle (z)\right|^{p}\frac{\prod_{i=1}^{r}\tau\left(z,v_{i},\delta_{\Omega}(z)\right)}{\sigma_{1}(z){\displaystyle \prod_{\substack{\{k\mbox{ \tiny such that }n-r+2\leq k\\
\mbox{\tiny and }k\leq n\}
}
}}\sigma_{k}(z)},\label{eq:punctual-norm-k-p}
\end{equation}
where $\tau\left(z,v_{i},\delta_{\Omega}(z)\right)$ is given by formula
(\ref{eq:definition-tau-general}).

Note that, if the coefficients of $f$ are continuous, then $\left\Vert f(z)\right\Vert _{\mathbbmss k,p}$
is also continuous.
\begin{rem*}
\quad\mynobreakpar
\begin{itemize}
\item If $\left(e_{i}\right)$ is a $\left(z,\delta(z)\right)$-extremal
basis (see \secref{Proofs-of-theorems}, after (\ref{eq:def-pseudo-balls})),
property (\ref{geometry-3}) of the geometry implies that (with the
notation (\ref{eq:definition-tau-extremal-basis}))
\begin{gather*}
\left|\left\langle f;\overline{v_{1}},\ldots,\overline{v_{r}}\right\rangle (z)\right|^{p}\prod_{i=1}^{r}\tau\left(z,v_{i},\delta_{\Omega}(z)\right)\lesssim\\
{\displaystyle \sideset{}{^{'}}\sum_{\substack{\left|I\right|=r\\
I=\left(i_{1},\ldots,i_{r}\right)
}
}}\left|\left\langle f(z);\overline{e_{i_{1}}},\ldots,\overline{e_{i_{r}}}\right\rangle \right|^{p}\prod_{k=1}^{r}\tau_{i_{k}}\left(z,\delta_{\Omega}(z)\right),
\end{gather*}
and (\ref{eq:equivalence-pseudoballs-polydisk}) and \lemref{Lemma-3-9}
show that, for all $k$, $\sigma_{k}(z)\simeq\tau_{k}\left(z,\delta_{\Omega}(z)\right)$
so
\[
\left\Vert f(z)\right\Vert _{\mathbbmss k,p}\simeq{\displaystyle \sideset{}{^{'}}\sum_{\substack{\left|I\right|=r\\
I=\left(i_{1},\ldots,i_{r}\right)
}
}}\left|\left\langle f(z);\overline{e_{i_{1}}},\ldots,\overline{e_{i_{r}}}\right\rangle \right|^{p}\frac{\prod_{k=1}^{r}\tau_{i_{k}}\left(z,\delta_{\Omega}(z)\right)}{\delta_{\Omega}(z){\displaystyle \prod_{\substack{\{k\mbox{ \tiny such that}\\n-r+2\leq k\\\mbox{\tiny and }k\leq n\}}}}\tau_{k}\left(z,\delta_{\Omega}(z)\right)}.
\]
\\
But, of course, in general, the second member of this last equivalence
is not a continuous function of $z$.
\item In \cite{CDMb}, for $p=1$, we defined $\left\Vert f(z)\right\Vert _{\mathbbmss k}=\sup_{\left\Vert v_{i}\right\Vert =1}\frac{\left|\left\langle f;\overline{v_{1}},\ldots,\overline{v_{r}}\right\rangle (z)\right|}{\sum_{i}\frac{\delta_{\Omega}(z)}{\tau\left(z,v_{i},\delta_{\Omega}(z)\right)}}$,
and, as before, (\ref{eq:equivalence-pseudoballs-polydisk}) and \lemref{Lemma-3-9}
show that this definition is equivalent to
\[
\sup_{\substack{I=\left(i_{1},\ldots,i_{r}\right)\\
i_{p}<i_{p+1}
}
}\left|\left\langle f(z);\overline{e_{i_{1}}},\ldots,\overline{e_{i_{r}}}\right\rangle \right|\min_{i\in I}\frac{\tau_{i}\left(z,\delta_{\partial\Omega}(z)\right)}{\delta_{\partial\Omega}(z)}.
\]
Thus, when $p=1$, the definition $\left\Vert f(z)\right\Vert _{\mathbbmss k,p}$
is equal to $\left\Vert f(z)\right\Vert _{\mathbbmss k}$ when $r=1$ and gives a smaller quantity when $r\geq2$.
So, when $r=p=1$ we will write indifferently $\left\Vert f(z)\right\Vert _{\mathbbmss k}$
or $\left\Vert f(z)\right\Vert _{\mathbbmss k,1}$.
\item In the case of strictly pseudo-convex domains it is clear that the
integrability of $\left\Vert f(z)\right\Vert _{\mathbbmss k,p}$ is
equivalent to the integrability of $\left|f\right|^{p}+\frac{\left|f\wedge\overline{\partial}\rho\right|^{p}}{\sqrt{-\rho}}$.
\end{itemize}
\end{rem*}

\medskip{}

\begin{stthm}
\label{thm:d-bar-for-Nev}For all $\alpha>0$ and all $p\in\left[1,+\infty\right[$,
there exists a constant $C>0$ such that, for all smooth $\overline{\partial}$-closed
$\left(0,r\right)$-form $f$, $1\leq r\leq n-1$, on $\overline{\Omega}$,
there exists a solution of the equation $\overline{\partial}u=f$,
continuous on $\overline{\Omega}$ such that
\[
\int_{\Omega}\left|u\right|^{p}\delta_{\Omega}^{\alpha-1}d\lambda\leq C\frac{1}{\min\left\{ \alpha,1\right\} }\int_{\Omega}\left\Vert f\right\Vert _{\mathbbmss k,p}\delta_{\Omega}^{\alpha}d\lambda.
\]
\end{stthm}

\begin{rem*}
Related (but non comparable) estimates can be found in \cite[Theorem 4.1]{AC00}
for strictly pseudoconvex domains and in \cite[Theorem 2.8]{Ale11}
for convex domains of finite type.
\end{rem*}

This result can be extended to $\alpha=0$ to get an $L^{p}(\partial\Omega)$
estimate (for $p=1$ this was done in \cite{CDMb}):
\begin{stthm}
\label{thm:d-bar-boundary-0-q-forms}For $p\in\left[1,+\infty\right[$
there exists a constant $C>0$ such that, for all smooth $\overline{\partial}$-closed
$\left(0,r\right)$-form $f$, $1\leq r\leq n-1$, on $\overline{\Omega}$,
there exists a solution of the equation $\overline{\partial}u=f$,
continuous on $\overline{\Omega}$ such that
\[
\int_{\partial\Omega}\left|u\right|^{p}d\sigma\leq C\int_{\Omega}\left\Vert f\right\Vert _{\mathbbmss k,p}d\lambda.
\]
\end{stthm}

\bigskip{}

For $r=1$ this result is weaker than the one which can be obtained
using Carleson measure of order $\beta$. Before stating our last
estimate let us recall these notions.

A measure $\mu$ in $\Omega$ is called a Carleson measure if
\[
\left\Vert \mu\right\Vert _{W^{1}}:=\sup_{z\in\partial\Omega,\,\varepsilon>0}\frac{\mu\left(P_{\varepsilon}(z)\cap\Omega\right)}{\sigma\left(P_{\varepsilon}(z)\cap\partial\Omega\right)}<+\infty,
\]
where $P_{\varepsilon}(z)$ is the extremal polydisk defined below
in (\ref{eq:def-polydisk}) and $\sigma$ the surface measure on $\partial\Omega$.
$W^{1}(\Omega)$ will denote the space of Carleson measures on $\Omega$.
Then, for $\beta\in\left]0,1\right[$ the space $W^{\beta}(\Omega)$
is the complex interpolated space between the space of bounded measures
on $\Omega$, denoted usually $W^{0}(\Omega)$, and $W^{1}(\Omega)$.
Moreover, we denote by $BMO(\partial\Omega)$ the BMO-space associated
to the anisotropic geometry on $\partial\Omega$. Then:
\begin{stthm}
\label{thm:Carleson-estimates-for-d-bar}There exists a constant $C>0$
such that, for all $p\in\left[1,+\infty\right[$ and all smooth $\overline{\partial}$-closed
$\left(0,1\right)$-form $f$ on $\overline{\Omega}$, there exists
a solution of the equation $\overline{\partial}u=f$, continuous on
$\overline{\Omega}$ such that:
\begin{itemize}
\item $\left\Vert u\right\Vert _{BMO(\partial\Omega)}\leq C\left\Vert \left\Vert f(\zeta)\right\Vert _{\mathbbmss k}d\lambda\right\Vert _{W^{1}}$;
\item $\left\Vert u\right\Vert _{L^{p}(\partial\Omega)}\leq C\left\Vert \left\Vert f(\zeta)\right\Vert _{\mathbbmss k}d\lambda\right\Vert _{W^{1-\nicefrac{1}{p}}}$.
\end{itemize}
\end{stthm}

This type of estimates where originally obtained, for strictly pseudoconvex
domains, by E. Amar \& A. Bonami in \cite{AB79} and extended to convex
domains of finite type by N. Nguyen in \cite{Ngu01} and W. Alexandre
in \cite{Ale11}.

\bigskip{}

A classical application of \thmref{d-bar-for-Nev} (for $p=1$ and
$r=1$) is the extension to lineally convex domains of the characterization
of the zero sets of the weighted Nevanlinna classes (called Nevanlinna-Djrbachian
classes in \cite{Cumenge-Navanlinna-convex}) obtained by A. Cumenge
for convex domains:
\begin{stthm}
A divisor $\mathcal{D}$ in $\Omega$ can be defined by a holomorphic
function $f$ satisfying $\int_{\Omega}\ln^{+}\left|f\right|\delta_{\Omega}^{\alpha-1}d\lambda<+\infty$,
$\alpha>0$, if and only if it satisfies the generalized Blaschke
condition $\int_{\mathcal{D}}\delta^{\alpha+1}d\lambda_{2n-2}<+\infty$.
\end{stthm}

As the proof of such result using \thmref{d-bar-for-Nev} is very
classical we will not give any detail in this paper.

\medskip{}

The two propositions \ref{prop:d-bar-gain-weight} and \ref{prop:d-bar-gain-exponent}
will be used to generalize some estimates obtained for weighted Bergman
projections of convex domains of finite type in \cite{CDM} using
the method introduced in \cite{CPDY}:
\begin{stthm}
\label{thm:estimates-bergman}Let $D$ be a smoothly bounded convex
domain of finite type $m$ in $\mathbb{C}^{n}$. Let $\chi$ be any
$\mathcal{C}^{2}$ non negative function in $\overline{D}$ which
is equivalent to the distance $\delta_{D}$ to the boundary of $D$
and let $\eta$ be a strictly positive $\mathcal{C}^{1}$ function
on $\overline{D}$. Let $P_{\omega}$ be the (weighted) Bergman projection
of the Hilbert space $L^{2}\left(D,\omega\right)$ where $\omega=\eta\chi^{r}$
with $r$ a non negative rational number.Then:
\begin{enumerate}
\item For $p\in\left]1,+\infty\right[$ and $-1<\beta\leq r$, $P_{\omega}$
maps continuously $L^{p}\left(D,\delta_{D}^{\beta}\right)$ into itself.
\item For $0<\alpha\leq\nicefrac{1}{m}$ $P_{\omega}$ maps continuously
the Lipschitz space $\Lambda_{\alpha}(D)$ into itself.
\end{enumerate}
\end{stthm}

This theorem combined with theorems \ref{thm:d-bar-q-gamma'-p-gamma-lip}
and \ref{thm:d-bar-for-Nev} extends to weighted situations the Corollary
1.3 of \cite{Cumenge-estimates-holder}
\begin{cor*}
Let $f$ a $\overline{\partial}$-closed $\left(0,1\right)$-form
on $D$. Under the assumptions and notations of \thmref{estimates-bergman},
the solution $u$ of the equation $\overline{\partial}u=f$ which
is orthogonal to holomorphic functions in $L^{2}(D,\omega)$ ($\omega=\eta\chi^{r}$) satisfies
the following estimates:
\begin{enumerate}
\item If the coefficients of $f$ are in $L^{p}(D,\delta_{D}^{\gamma})$,
$-1<\gamma$ then:

\begin{enumerate}
\item $u\in L^{q}(D,\delta_{D}^{\gamma'})$, with $\frac{1}{q}=\frac{1}{p}-\frac{1-m(\gamma-\gamma')}{m(\gamma'+n)+2}$
and $-1<\gamma'\leq r$, $\gamma-\nicefrac{1}{m}\leq\gamma'\leq\gamma$,
if $1\leq p<\frac{m(\gamma'+n)+2}{1-m(\gamma-\gamma')}$, and $q>1$;
\item $u\in\Lambda_{\alpha}(D)$, with $\alpha=\frac{1}{m}\left[1-\frac{m(\gamma+n)+2}{p}\right]$,
if $p\in\left]m(\gamma+n)+2,+\infty\right]$.
\end{enumerate}
\item If $\left\Vert f\right\Vert _{\mathbbmss k,p}\delta_{D}^{\gamma+1}$
is in $L^{1}(D)$, $-1<\gamma\leq r$, then $u\in L^{p}\left(D,\delta_{D}^{\gamma}\right)$.
\end{enumerate}
\end{cor*}

\section{\label{sec:Proofs-of-theorems}Proofs of theorems \ref{thm:d-bar-q-gamma'-p-gamma-lip}
to \ref{thm:Carleson-estimates-for-d-bar}}

First, by standard regularization procedure, it suffices to prove
theorems \ref{thm:d-bar-q-gamma'-p-gamma-lip}, and \ref{thm:d-bar-for-Nev}
for forms smooth in $\overline{\Omega}$.

To solve the $\overline{\partial}$-equation on a lineally convex
domain of finite type, we use exactly the method introduced in \cite{CDMb},
except for the proof of \thmref{Carleson-estimates-for-d-bar} where
a modification of the form $s(z,\zeta)$ is done. We now briefly recall
the notations and main results from that work.

If $f$ is a smooth $\left(0,r\right)$-form $\overline{\partial}$-closed,
the then
\begin{equation}
f(z)=\left(-1\right)^{r+1}\overline{\partial_{z}}\left(\int_{\Omega}f(\zeta)\wedge K_{N}^{1}(z,\zeta)\right)-\int_{\Omega}f(\zeta)\wedge P_{N}(z,\zeta),\label{eq:And-Ber-formula}
\end{equation}
where $K_{N}^{1}$ (resp. $P_{N}$) is the component of a kernel $K_{N}$
(formula (2.7) of \cite{CDMb}) of bi-degree $\left(0,r\right)$ in
$z$ and $\left(n,n-r-1\right)$ in $\zeta$ (resp. $\left(0,r\right)$
in $z$ and $\left(n,n-r\right)$ in $\zeta$) constructed with the
method of \cite{BA82} using the Diederich-Fornaess support function
constructed in \cite{Diederich-Fornaess-Support-Func-lineally-cvx}
(see also Theorem 2.2 of \cite{CDMb}) and the function $G(\xi)=\frac{1}{\xi^{N}}$
with a sufficiently large number $N$ (instead of $G(\xi)=\frac{1}{\xi}$
in formula (2.7) of \cite{CDMb}).

Then, the form $\int_{\Omega}f(\zeta)\wedge P_{N}(z,\zeta)$ is $\overline{\partial}$-closed
and the operator $T$ solving the $\overline{\partial}$-equation
in theorems \ref{thm:d-bar-q-gamma'-p-gamma-lip} and \ref{thm:d-bar-for-Nev}
is defined on smooth forms by
\[
Tf(z)=\left(-1\right)^{r+1}\int_{\Omega}f(\zeta)\wedge K_{N}^{1}(z,\zeta)-\overline{\partial}^{*}\mathcal{N}\left(\int_{\Omega}f(\zeta)\wedge P_{N}(z,\zeta)\right),
\]
where $\overline{\partial}^{*}\mathcal{N}$ is the canonical solution
of the $\overline{\partial}$-equation derived from the theory of
the $\overline{\partial}$-Neumann problem on pseudoconvex domains
of finite type.

This formula is justified by the fact that, when the coefficients
of $f$ are in $L^{1}\left(\Omega,\delta_{\Omega}^{\gamma}\right)$
($\gamma>-1$) then, given a large integer $s$, if $N$ is chosen
sufficiently large, the coefficients of the form $\int_{\Omega}f(\zeta)\wedge P_{N}(z,\zeta)$
are in the Sobolev space $L_{s}^{2}(\Omega)$. More precisely, it
is clear that lemmas 2.2 and 2.3 of \cite{CDMb} remains true with
weighted estimates depending on the choice of $N$:
\begin{sllem}
For $r\geq1$ and $\gamma\leq N$, all the $z$-derivatives of $P_{N}(z,\zeta)\left(-\rho(\zeta)\right)^{-\gamma}$
are uniformly bounded in $\overline{\Omega}\times\overline{\Omega}$,
and, for each positive integer $s$, there exists a constant $C_{s,N,\gamma}$
such that, if $f$ is $\left(0,r\right)$-form with coefficients in
$L^{1}(\Omega,\delta_{\Omega}^{\gamma})$,
\[
\left\Vert \int_{\Omega}f(\zeta)\wedge P_{N}(z,\zeta)\right\Vert _{L_{s}^{2}(\Omega)}\leq C_{s,N,\gamma}\left\Vert f\right\Vert _{L^{1}(\Omega,\delta_{\Omega}^{\gamma})}.
\]

\end{sllem}

As $\Omega$ is assumed to be smooth and of finite type, the regularity
results of the $\overline{\partial}$-Neumann problem (\cite{Kohn-Nirenberg-1965}
and \cite{Catlin-Est.-Sous-ellipt.}) imply:
\begin{sllem}
For $r\geq1$ and $-1<\gamma\leq N$, for each positive integer $s$,
if $f$ is a $\overline{\partial}$-closed $\left(0,r\right)$-form with coefficients in $L^{1}(\Omega,\delta_{\Omega}^{\gamma})$
and $g=\int_{\Omega}f(\zeta)\wedge P_{N}(z,\zeta)$, then $\overline{\partial}^{*}\mathcal{N}(g)$
is a solution of the equation $\overline{\partial}u=g$ satisfying
$\left\Vert \overline{\partial}^{*}\mathcal{N}(g)\right\Vert _{L_{s}^{2}(\Omega)}\leq C_{s,N,\gamma}\left\Vert f\right\Vert _{L^{1}(\Omega,\delta^{\gamma})}$.
\end{sllem}

Applying Sobolev lemma we immediately get:
\begin{sllem}
\label{lem:3-3-C1-regularity-PN}For $r\geq1$ and $-1<\gamma\leq N$,
if $f$ is a $\overline{\partial}$-closed $\left(0,r\right)$-form with coefficients in $L^{1}(\Omega,\delta_{\Omega}^{\gamma})$
and $g=\int_{\Omega}f(\zeta)\wedge P_{N}(z,\zeta)$, then $\overline{\partial}^{*}\mathcal{N}(g)$
is a solution of the equation $\overline{\partial}u=g$ satisfying
$\left\Vert \overline{\partial}^{*}\mathcal{N}(g)\right\Vert _{\mathcal{C}^{1}(\overline{\Omega})}\leq C\left\Vert f\right\Vert _{L^{1}\left(\Omega,\delta^{\gamma}\right)}$.
\end{sllem}

\medskip{}

Finally the proofs of our theorems are reduced to the proofs of good
estimates for the operator $T_{K}$ defined by
\begin{equation}
T_{K}:\,f\mapsto\int_{\Omega}f(\zeta)\wedge K_{N}^{1}(z,\zeta).\label{eq:operator_K1}
\end{equation}

To do it we need to recall the anisotropic geometry of $\Omega$ and
the basic estimates given in \cite{CDMb}.

For $\zeta$ close to $\partial\Omega$ and $\varepsilon\leq\varepsilon_{0}$,
$\varepsilon_{0}$ small, define, for all unitary vector $v$,
\begin{equation}
\tau\left(\zeta,v,\varepsilon\right)=\sup\left\{ c\mbox{ such that }\rho\left(\zeta+\lambda v\right)-\rho(\zeta)<\varepsilon,\,\forall\lambda\in\mathbb{C},\,\left|\lambda\right|<c\right\} .\label{eq:definition-tau-general}
\end{equation}
Note that the lineal convexity hypothesis implies that the function
$\left(\zeta,\varepsilon\right)\mapsto\tau(\zeta,v,\varepsilon)$
is smooth. In particular, $\zeta\mapsto\tau(\zeta,v,\delta_{\Omega}(\zeta))$
is a smooth function. The pseudo-balls $B_{\varepsilon}(\zeta)$ (for
$\zeta$ close to the boundary of $\Omega$) of the homogeneous space
associated to the anisotropic geometry of $\Omega$ are
\begin{equation}
B_{\varepsilon}(\zeta)=\left\{ \xi=\zeta+\lambda u\mbox{ with }\left|u\right|=1\mbox{ and }\left|\lambda\right|<c_{0}\tau(\zeta,u,\varepsilon)\right\} \label{eq:def-pseudo-balls}
\end{equation}
where $c_{0}$ is chosen sufficiently small depending only on the
defining function $\rho$ of $\Omega$.

Let $\zeta$ and $\varepsilon$ be fixed. Then, an orthonormal basis
$\left(v_{1},v_{2},\ldots,v_{n}\right)$ is called \emph{$\left(\zeta,\varepsilon\right)$-extremal}
(or $\varepsilon$-\emph{extremal}, or simply \emph{extremal}) if
$v_{1}$ is the complex normal (to $\rho$) at $\zeta$, and, for
$i>1$, $v_{i}$ belongs to the orthogonal space of the vector space
generated by $\left(v_{1},\ldots,v_{i-1}\right)$ and minimizes $\tau\left(\zeta,v,\varepsilon\right)$
in the unit sphere of that space. In association to an extremal basis, we denote
\begin{equation}
\tau(\zeta,v_{i},\varepsilon)=\tau_{i}(\zeta,\varepsilon).\label{eq:definition-tau-extremal-basis}
\end{equation}

Then we defined polydiscs $AP_{\varepsilon}(\zeta)$ by
\begin{equation}
AP_{\varepsilon}(\zeta)=\left\{ z=\zeta+\sum_{k=1}^{n}\lambda_{k}v_{k}\mbox{ such that }\left|\lambda_{k}\right|\leq c_{0}A\tau_{k}(\zeta,\varepsilon)\right\} .\label{eq:def-polydisk}
\end{equation}

$P_{\varepsilon}(\zeta)$ being the corresponding polydisc with $A=1$
and we also define
\[
d(\zeta,z)=\inf\left\{ \varepsilon\mbox{ such that }z\in B_{\varepsilon}(\zeta)\right\} ,
\]
so
\[
d(\zeta,z)\simeq\inf\left\{ \varepsilon\mbox{ such that }z\in P_{\varepsilon}(\zeta)\right\} .
\]

\begin{rem*}
Note that there is neither unicity of the extremal basis $\left(v_{1},v_{2},\ldots,v_{n}\right)$
nor of associated polydisk $P_{\varepsilon}(\zeta)$. However the
polydisks associated to two different $\left(\zeta,\varepsilon\right)$-extremal
basis are equivalent. Thus in all the paper $P_{\varepsilon}(\zeta)$
will denote a polydisk associated to any $\left(\zeta,\varepsilon\right)$-extremal
basis and $\tau_{i}(\zeta,\varepsilon)$ the radius of $P_{\varepsilon}(\zeta)$.
\end{rem*}

The fundamental result here is that $d$ is a pseudo-distance which
means that there exists a constant $K$ and, $\forall\alpha>0$, constants
$c(\alpha)$ and $C(\alpha)$ such that
\begin{equation}
\mbox{for }\zeta\in P_{\varepsilon}(z),\,P_{\varepsilon}(z)\subset P_{K\varepsilon}(\zeta),\label{eq:inclusion-polydisques-pseudo}
\end{equation}
and
\begin{equation}
c(\alpha)P_{\varepsilon}(\zeta)\subset P_{\alpha\varepsilon}(\zeta)\subset C(\alpha)P_{\varepsilon}(\zeta)\mbox{ and }P_{c(\alpha)\varepsilon}(\zeta)\subset\alpha P_{\varepsilon}(\zeta)\subset P_{C(\alpha)\varepsilon}(\zeta).\label{eq:polydiscs-pseudodistance}
\end{equation}

Moreover the pseudo-balls $B_{\varepsilon}$ and the polydiscs $P_{\varepsilon}$
are equivalent in the sense that there exists a constant $K>0$ depending
only on $\Omega$ such that
\begin{equation}
\frac{1}{K}P_{\varepsilon}(\zeta)\subset B_{\varepsilon}(\zeta)\subset KP_{\varepsilon}(\zeta).\label{eq:equivalence-pseudoballs-polydisk}
\end{equation}

For $\zeta$ close to $\partial\Omega$ and $\varepsilon>0$ small,
the basic properties of this geometry are (see \cite{Conrad_lineally_convex}
and \cite{CDMb}):
\begin{enumerate}
\item \label{geometry-1}Let $w=\left(w_{1},\ldots,w_{n}\right)$ be an
orthonormal system of coordinates centered at $\zeta$. Then
\[
\left|\frac{\partial^{\left|\alpha+\beta\right|}\rho(\zeta)}{\partial w^{\alpha}\partial\bar{w}^{\beta}}\right|\lesssim\frac{\varepsilon}{\prod_{i}\tau\left(\zeta,w_{i},\varepsilon\right)^{\alpha_{i}+\beta_{i}}},\,\left|\alpha+\beta\right|\geq1.
\]

\item \label{geometry-2}Let $\nu$ be a unit vector. Let $a_{\alpha\beta}^{\nu}(\zeta)=\frac{\partial^{\alpha+\beta}\rho}{\partial\lambda^{\alpha}\partial\bar{\lambda}^{\beta}}\left(\zeta+\lambda\nu\right)_{|\lambda=0}$.
Then
\[
\sum_{1\leq\left|\alpha+\beta\right|\leq2m}\left|a_{\alpha\beta}^{\nu}(\zeta)\right|\tau(\zeta,\nu,\varepsilon)^{\alpha+\beta}\simeq\varepsilon,
\]
where $m$ is the type of $\Omega$.
\item \label{geometry-3}If $\left(v_{1},\ldots,v_{n}\right)$ is a $\left(\zeta,\varepsilon\right)$-extremal
basis and $\gamma=\sum_{1}^{n}a_{j}v_{j}\neq0$, then
\[
\frac{1}{\tau(\zeta,\gamma,\varepsilon)}\simeq\sum_{j=1}^{n}\frac{\left|a_{j}\right|}{\tau_{j}(\zeta,\varepsilon)}.
\]

\item \label{geometry-4}If $v$ is a unit vector then:

\begin{enumerate}
\item $z=\zeta+\lambda v\in P_{\varepsilon}(\zeta)$ implies $\left|\lambda\right|\lesssim\tau(\zeta,v,\varepsilon)$,
\item $z=\zeta+\lambda v$ with $\left|\lambda\right|\leq\tau(\zeta,v,\varepsilon)$
implies $z\in CP_{\varepsilon}(\zeta)$.
\end{enumerate}
\item \label{geometry-5}If $\nu$ is the unit complex normal, then $\tau(\zeta,v,\varepsilon)=\varepsilon$
and if $v$ is any unit vector and $\lambda\geq1$,
\begin{equation}
\lambda^{\nicefrac{1}{m}}\tau(\zeta,v,\varepsilon)\lesssim\tau(\zeta,v,\lambda\varepsilon)\lesssim\lambda\tau(\zeta,v,\varepsilon),\label{eq:comp-tau-epsilon-tau-lambda-epsilon}
\end{equation}
where $m$ is the type of $\Omega$.\end{enumerate}
\begin{sllem}
\label{lem:3.4-maj-deriv-rho-equiv-tho-i-z-zeta}For $z$ close to
$\partial\Omega$, $\varepsilon$ small and $\zeta\in P_{\varepsilon}(z)$
or $z\in P_{\varepsilon}(\zeta)$, we have, for all $1\leq i\leq n$:
\begin{enumerate}
\item $\tau_{i}(z,\varepsilon)=\tau\left(z,v_{i}\left(z,\varepsilon\right),\varepsilon\right)\simeq\tau\left(\zeta,v_{i}\left(z,\varepsilon\right),\varepsilon\right)$
where $\left(v_{i}\left(z,\varepsilon\right)\right)_{i}$ is the $\left(z,\varepsilon\right)$-extremal
basis;
\item $\tau_{i}(\zeta,\varepsilon)\simeq\tau_{i}(z,\varepsilon)$;
\item In the coordinate system $\left(z_{i}\right)$ associated to the $\left(z,\varepsilon\right)$-extremal
basis, $\left|\frac{\partial\rho}{\partial z_{i}}(\zeta)\right|\lesssim\frac{\varepsilon}{\tau_{i}}$
where $\tau_{i}$ is either $\tau_{i}\left(z,\varepsilon\right)$
or $\tau_{i}\left(\zeta,\varepsilon\right)$.
\end{enumerate}
\end{sllem}

\begin{proof}
(1) is proved in \cite{Conrad_lineally_convex} (together with the
properties of the geometry). (2) follows the properties of the geometry
((\ref{eq:inclusion-polydisques-pseudo}) and (\ref{eq:polydiscs-pseudodistance}))
and formula (\ref{eq:integral-mod(z-zeta)-inverse-power-1+mu}) of
\lemref{Lemma-3-9} and (3) is a consequence of (1), (2) and the first
property of the geometry.\end{proof}
\begin{rem*}
In (1) above $\tau\left(\zeta,v_{i}\left(z,\varepsilon\right),\varepsilon\right)$
is not $\tau_{i}\left(\zeta,\varepsilon\right)$ because the extremal
basis at $z$ and $\zeta$ are different but (2) implies that these
quantities are equivalent.
\end{rem*}

\medskip{}

We now recall the detailed expression of $K_{N}^{1}$ (\cite{CDMb}
sections 2.2 and 2.3):
\[
K_{N}^{1}(z,\zeta)=\sum_{k=n-r}^{n-1}C'_{k}\frac{\rho(\zeta)^{k+N}s\wedge\left(\partial_{\bar{\zeta}}Q\right)^{n-r}\wedge\left(\partial_{\bar{z}}Q\right)^{k+r-n}\wedge\left(\partial_{\bar{z}}s\right)^{n-k-1}}{\left\langle s(z,\zeta),\zeta-z\right\rangle ^{2\left(n-k\right)}\left(\frac{1}{K_{0}}S(z,\zeta)+\rho(\zeta)\right)^{k+N}},
\]
where $s$ is a $\left(1,0\right)$-form satisfying
\[
c\left|z-\zeta\right|^{2}\leq\left|\left\langle s(z,\zeta),\zeta-z\right\rangle \right|\leq C\left|\zeta-z\right|
\]
uniformly for $\zeta\in\overline{\Omega}$ and $z$ in any compact
subset of $\Omega$, and
\[
Q(z,\zeta)=\frac{1}{K_{0}\rho(\zeta)}\sum_{i=1}^{n}Q_{i}(z,\zeta)d\left(\zeta_{i}-z_{i}\right)
\]
with
\begin{equation}
S(z,\zeta)=\chi(z,\zeta)S_{0}(z,\zeta)-\left(1-\chi(z,\zeta)\right)\left|z-\zeta\right|^{2}=\sum_{i=1}^{n}Q_{i}(z,\zeta)\left(z_{i}-\zeta_{i}\right),\label{eq:def-S-and-Qi}
\end{equation}
$S_{0}$ being the holomorphic support function of Diederich-Fornaess
(see \cite{Diederich-Fornaess-Support-Func-lineally-cvx} or Theorem
2.2 of \cite{CDMb}) and $\chi$ a truncating function which is equal
to $1$ when both $\left|z-\zeta\right|$ and $\delta_{\Omega}(\zeta)$
are small and $0$ if one of these expressions is large (see the beginning
of Section 2.2 of \cite{CDMb} for a precise definition). Recall that
$K_{0}$ is chosen so that
\[
\Re\mathrm{e}\left(\rho(\zeta)+\frac{1}{K_{0}}S(z,\zeta)\right)<\frac{\rho(\zeta)}{2},
\]
which implies
\begin{equation}
\left|\rho(\zeta)+\frac{1}{K_{0}}S(z,\zeta)\right|\gtrsim\left|\rho(\zeta)\right|.\label{eq:real-part-Q-zeta-z-plus-1}
\end{equation}

The precise choice of the form $s$ is
\begin{equation}
s(z,\zeta)=\sum_{i=1}^{n}\left(\overline{\zeta_{i}}-\overline{z_{i}}\right)d\left(\zeta_{i}-z_{i}\right)\label{eq:basic-choice-for-s(z,zeta)}
\end{equation}
for all the proofs except for the proof of \thmref{Carleson-estimates-for-d-bar}
where a different choice is needed.

The following estimates of the expressions appearing in $K_{N}^{1}$
are basic (see \cite{CDMb}):
\begin{sllem}
\label{lem:3.5-lemma-maj-rho-S}For $\zeta\in P_{2\varepsilon}(z)\setminus P_{\varepsilon}(z)$
or $z\in P_{2\varepsilon}(\zeta)\setminus P_{\varepsilon}(\zeta)$,
we have:

\[
\left|\rho(\zeta)+\frac{1}{K_{0}}S(z,\zeta)\right|\gtrsim\varepsilon,\,\left(z,\zeta\right)\in\bar{\Omega}\times\bar{\Omega}.
\]

\end{sllem}

\begin{sllem}
\label{lem:3.6-Lemma-maj-deriv_rho_Q}For $z_{0}$ close to $\partial\Omega$,
$\varepsilon$ small and $z,\,\zeta\in P_{\varepsilon}(z_{0})$, in
the coordinate system $\left(\zeta_{i}\right)$ associated to the
$\left(z_{0},\varepsilon\right)$-extremal basis, we have:
\begin{enumerate}
\item $\left|Q_{i}(z,\zeta)\right|+\left|Q_{i}(\zeta,z)\right|\lesssim\frac{\varepsilon}{\tau_{i}}$,
\item $\left|\frac{\partial Q_{i}(z,\zeta)}{\partial\overline{\zeta_{j}}}\right|+\left|\frac{\partial Q_{i}(z,\zeta)}{\partial\zeta_{j}}\right|+\left|\frac{\partial Q_{i}(z,\zeta)}{\partial z_{j}}\right|\lesssim\frac{\varepsilon}{\tau_{i}\tau_{j}}$,
\item $\left|\frac{\partial^{2}Q_{i}(z,\zeta)}{\partial\overline{\zeta}_{j}\partial z_{k}}\right|+\left|\frac{\partial^{2}Q_{i}(z,\zeta)}{\partial\overline{\zeta}_{j}\partial\overline{z}_{k}}\right|\lesssim\frac{\varepsilon}{\tau_{i}\tau_{j}\tau_{k}}$,
\end{enumerate}
where $\tau_{i}$ are either $\tau_{i}(z,\varepsilon)$, $\tau_{i}(\zeta,\varepsilon)$
or $\tau_{i}\left(z_{0},\varepsilon\right)$.\end{sllem}

\begin{proof}
This Lemma follows \cite{Diederich-Fischer_Holder-linally-convex}
and \lemref{3.4-maj-deriv-rho-equiv-tho-i-z-zeta}.
\end{proof}
\bigskip{}

The preceding lemmas and the properties of the geometry easily give
the following estimates of the kernel $K_{N}^{1}$:
\begin{sllem}
For $\varepsilon$ small enough and $z$ sufficiently close to the
boundary, with the choice (\ref{eq:basic-choice-for-s(z,zeta)}) for
$s$, we have:

If $\zeta\in P_{\varepsilon}(z)$ or $z\in P_{\varepsilon}(\zeta)$,
\[
\left|K_{N}^{1}(z,\zeta)\right|\lesssim\frac{\left|\rho(\zeta)\right|^{N-1}\left(\left|\rho(\zeta)\right|+\varepsilon\right)\varepsilon^{n-r}}{\prod_{i=1}^{n-r}\tau_{i}^{2}\left|\frac{1}{K_{0}}S(z,\zeta)+\rho(\zeta)\right|^{N+n-r}}\frac{1}{\left|z-\zeta\right|^{2r-1}},
\]
where $\tau_{i}$ is $\tau_{i}(z,\varepsilon)$ or $\tau_{i}(\zeta,\varepsilon)$.\end{sllem}

\begin{proof}
Indeed, under the conditions of the lemma, $Q(z,\zeta)=\frac{1}{K_{0}\rho(\zeta)}\sum Q_{i}(z,\zeta)d\left(\zeta_{i}-z_{i}\right)$
is holomorphic in $z$ and $K_{N}^{1}$ is reduced to
\[
K_{N}^{1}(z,\zeta)=c\frac{\rho(\zeta)^{n+N-r}\wedge\left(d_{\overline{\zeta}}Q\right)^{n-r}\wedge\left(\partial_{\overline{z}}s\right)^{r-1}}{\left|z-\zeta\right|^{2r}\left(\frac{1}{K_{0}}S(z,\zeta)+\rho(\zeta)\right)^{N+n-r}}.
\]

\end{proof}
In particular:
\begin{sllem}
\label{lem:3.8-basic-estimates-kernel}For $\varepsilon$ small enough
and $z$ sufficiently close to the boundary and $s$ given by (\ref{eq:basic-choice-for-s(z,zeta)}):
\begin{enumerate}
\item If $\varepsilon\leq\delta_{\partial\Omega}(z)$, for $\zeta\in P_{\varepsilon}(z)$
or $z\in P_{\varepsilon}(\zeta)$,
\[
\left|K_{N}^{1}(z,\zeta)\right|\lesssim\frac{1}{\prod_{i=1}^{n-r}\tau_{i}^{2}}\frac{1}{\left|z-\zeta\right|^{2r-1}},
\]
where $\tau_{i}$ is $\tau_{i}(z,\varepsilon)$ or $\tau_{i}(\zeta,\varepsilon)$.
\end{enumerate}
For $1+r-n\leq k\leq N+n-1$ and $\zeta\in P_{2\varepsilon}(z)\setminus P_{\varepsilon}(z)$
or $z\in P_{2\varepsilon}(\zeta)\setminus P_{\varepsilon}(\zeta)$:
\begin{enumerate}{\SEI{1}}
\item 
\[
\left|K_{N}^{1}(z,\zeta)\right|\lesssim\frac{\left|\rho(\zeta)\right|^{k}}{\varepsilon^{k}}\frac{1}{\prod_{i=1}^{n-r}\tau_{i}^{2}}\frac{1}{\left|z-\zeta\right|^{2r-1}},
\]

\item 
\[
\left|\nabla_{z}K_{N}^{1}(z,\zeta)\right|\lesssim\frac{\left|\rho(\zeta)\right|^{k}}{\varepsilon^{k+1}}\frac{1}{\prod_{i=1}^{n-r}\tau_{i}^{2}}\frac{1}{\left|z-\zeta\right|^{2r-1}},
\]

\end{enumerate}
where $\tau_{i}$ is either $\tau_{i}(z,\varepsilon)$ or $\tau_{i}(\zeta,\varepsilon)$.
\end{sllem}

\begin{sllem}
\label{lem:Lemma-3-9}For $z\in\Omega$, close to $\partial\Omega$,
$\delta$ small, $r\in\left\{ 1,\ldots,n-1\right\} $ and $0\leq\mu<1$,
\begin{equation}
\int_{P_{\delta}(z)}\frac{d\lambda(\zeta)}{\left|z-\zeta\right|^{2r-1+\mu}}\simeq\tau_{n-r+1}(z,\delta)^{1-\mu}\prod_{j=1}^{n-r}\tau_{j}^{2}(z,\delta),\label{eq:integral-mod(z-zeta)-inverse-power-1+mu}
\end{equation}
and, for $\alpha>0$,
\begin{equation}
\int_{P_{\delta}(\zeta)}\frac{\delta_{\Omega}(z)^{\alpha-1}}{\left|z-\zeta\right|^{2r-1}}d\lambda(z)\lesssim\begin{cases}
\frac{\delta^{\alpha-1}}{\alpha}\tau_{n-r+1}(\zeta,\delta)\prod_{j=1}^{n-r}\tau_{j}^{2}(\zeta,\delta), & \mbox{if }\alpha\leq1,\\
\left(\delta_{\Omega}(\zeta)+\delta\right)^{\alpha-1}\tau_{n-r+1}(\zeta,\delta)\prod_{j=1}^{n-r}\tau_{j}^{2}(\zeta,\delta), & \mbox{if }\alpha>1.
\end{cases}.\label{eq:integral-mod(z-zeta)-incerse-delta-alpha-1}
\end{equation}
\end{sllem}

\begin{proof}
Let us prove (\ref{eq:integral-mod(z-zeta)-inverse-power-1+mu}).
Let 
\[
D=\left\{ \zeta\mbox{ such that }\left|\zeta_{i}-z_{i}\right|\leq\tau_{i}(z,\delta),\,i\leq n-r\mbox{ and }\left|\zeta_{i}-z_{i}\right|\lesssim\tau_{n-r+1}(z,\delta),\,i>n-r\right\} .
\]
Then the volume of $D$ is $V(D)\simeq\prod_{j=1}^{n-r}\tau_{j}(z,\delta)^{2}\tau_{n-r+1}(z,\delta)^{2r}$
and, on $D$, $\frac{1}{\left|\zeta-z\right|^{2r-1+\mu}}\gtrsim\left(\frac{1}{\tau_{n-r+1}(z,\delta)}\right)^{2r-1+\mu}$.
This proves the inequality $\gtrsim$.

To prove the converse inequality let us first consider the sets
\[
E_{l}=\left\{ \left|\zeta_{i}-z_{i}\right|\leq\tau_{i}(z,\delta),\,i\leq n-r\mbox{ and }2^{-l}\tau_{n-r+1}(z,\delta)\leq\left|\zeta',-z'\right|\lesssim2^{-l+1}\tau_{n-r+1}(z,\delta)\right\} 
\]
where $\zeta'=\left(\zeta_{n-r+1},\ldots,\zeta_{n}\right)$ and $z'=\left(z_{n-r+1},\ldots,z_{n}\right)$.
Then the volume of $E_{l}$ is 
\[
V(E_{l})\lesssim\prod_{j=1}^{n-r}\tau_{j}^{2}(z,\delta)\left(2^{-l}\tau_{n-r+1}(z,\delta)\right)^{2r},
\]
and, for $\zeta\in E_{l}$,
\[
\frac{1}{\left|\zeta-z\right|^{2r-1+\mu}}\lesssim\left(\frac{1}{2^{-l}\tau_{n-r+1}(z,\delta)}\right)^{2r-1+\mu}.
\]

Thus
\[
\int_{E_{l}}\frac{d\lambda(\zeta)}{\left|z-\zeta\right|^{2r-1+\mu}}\lesssim\prod_{i=1}^{n-r}\tau_{i}^{2}(z,\delta)\tau_{n-r+1}(z,\delta)^{1-\mu}\left(2^{-l}\right)^{1-\mu},
\]
proving the inequality if $r=1$. If $r\geq2$, consider now the sets
\[
F_{l}=\left\{ \left|\zeta_{i}-z_{i}\right|\leq\tau_{i}(z,\delta),\,i\leq n-r+1\mbox{ and }2^{l}\tau_{n-r+1}(z,\delta)\leq\left|\zeta",-z"\right|\lesssim2^{l+1}\tau_{n-r+1}(z,\delta)\right\} 
\]
where $\zeta"=\left(\zeta_{n-r+2},\ldots,\zeta_{n}\right)$ and $z"=\left(z_{n-r+2},\ldots,z_{n}\right)$.
Then the volume of $F_{l}$ is 
\[
V(F_{l})=\lesssim\prod_{j=1}^{n-r+1}\tau_{j}(z,\delta)^{2}\left(2^{l}\tau_{n-r+1}(z,\delta)\right)^{2r-2},
\]
and, for $\zeta\in F_{l}$,
\[
\frac{1}{\left|\zeta-z\right|^{2r-1+\mu}}\lesssim\left(\frac{1}{2^{l}\tau_{n-r+1}(z,\delta)}\right)^{2r-1+\mu}.
\]

Thus
\[
\int_{F_{l}}\frac{d\lambda(\zeta)}{\left|z-\zeta\right|^{2r-1+\mu}}\lesssim\prod_{i=1}^{n-r}\tau_{i}(z,\delta)^{2}\tau_{n-r+1}(z,\delta)^{1-\mu}\left(2^{l}\right)^{-1+\mu},
\]
finishing the proof of (\ref{eq:integral-mod(z-zeta)-inverse-power-1+mu}).

(\ref{eq:integral-mod(z-zeta)-incerse-delta-alpha-1}) is proved similarly
considering the sets
\[
E_{l}=P_{\delta}(\zeta)\cap\left\{ z\mbox{ such that }\delta_{\Omega}(z)\in\left[\frac{1}{2^{l+1}}\delta_{\Omega}(\zeta),\frac{1}{2^{l}}\delta_{\Omega}(\zeta)\right]\right\} .
\]

\end{proof}

We now detail the different proofs of the theorems in the next sub-sections.
Recall that, by \lemref{3-3-C1-regularity-PN}, we only have to prove
the estimates for the operator $T_{K}$ associated to the kernel $K_{N}^{1}$.

\subsection{Proof of Theorem \ref{thm:d-bar-q-gamma'-p-gamma-lip}}

\quad

In this proof the form $s(z,\zeta)$ is given by the formula (\ref{eq:basic-choice-for-s(z,zeta)}).
\begin{proof}[Proof of (1) of \thmref{d-bar-q-gamma'-p-gamma-lip}]
It is based on a version of a classical operator estimate which can
be found, for example, in Appendix B of the book of M. Range \cite{RM86}:
\begin{sllem}
Let $\Omega$ be a smoothly bounded domain in $\mathbb{C}^{n}$. Let
$\mu$ and $\nu$ be two positive measures on $\Omega$. Let $K$
be a measurable function on $\Omega\times\Omega$. Assume that there
exists a positive number $\varepsilon_{0}>0$, a positive constant
$C$ and a real number $s\geq1$ such that:
\begin{enumerate}
\item $\int_{\Omega}\left|K(z,\zeta)\right|^{s}\delta_{\Omega}(\zeta)^{-\varepsilon}d\mu(\zeta)\leq C\delta_{\Omega}(z)^{-\varepsilon}$,
\item $\int_{\Omega}\left|K(z,\zeta)\right|^{s}\delta_{\Omega}(z)^{-\varepsilon}d\nu(z)\leq C\delta_{\Omega}(\zeta)^{-\varepsilon}$,
\end{enumerate}
for all $0<\varepsilon\leq\varepsilon_{0}$, where $\delta_{\Omega}$
denotes the distance to the boundary of $\Omega$. Then the linear
operator $T$ defined by
\[
Tf(z)=\int_{\Omega}K(z,\zeta)f(\zeta)d\mu(\zeta)
\]
 is bounded from $L^{p}\left(\Omega,\mu\right)$ to $L^{q}\left(\Omega,\nu\right)$
for all $1\leq p,q<\infty$ such that $\frac{1}{q}=\frac{1}{p}+\frac{1}{s}-1$.\end{sllem}

\begin{proof}[Short proof]
This is exactly the proof given by M. Range in his book: let $\varepsilon$
be sufficiently small. Writing 
\[
\left|Kf\right|=\left(|K|^{s}|f|^{p}\delta_{\Omega}(\zeta)^{\varepsilon\frac{p-1}{p}q}\right)^{\nicefrac{1}{q}}\left(|K|^{1-\frac{s}{q}}\delta_{\Omega}(\zeta)^{-\varepsilon\frac{p-1}{p}}\right)|f|^{1-\frac{p}{q}},
\]
H\"older's inequality (with $\frac{1}{q}+\frac{p-1}{p}+\frac{s-1}{s}=1$)
gives
\begin{multline*}
\left|Tf(z)\right|\leq\left(\int_{\Omega}\left|K(z,\zeta)\right|^{s}\delta_{\Omega}(\zeta)^{\varepsilon\frac{p-1}{p}q}\left|f\right|^{p}(\zeta)d\mu(\zeta)\right)^{\nicefrac{1}{q}}\\
\left(\int_{\Omega}\left|K(z,\zeta)\right|^{s}\delta_{\Omega}(\zeta)^{-\varepsilon}\right)^{\frac{p-1}{p}}\left(\int_{\Omega}\left|f(\zeta)\right|^{p}d\mu(\zeta)\right)^{\frac{s-1}{s}}.
\end{multline*}
The first hypothesis of the lemma gives (for $\varepsilon\leq\varepsilon_{0}$)
\begin{multline*}
\left|Tf(z)\right|^{q}\leq C\left(\int_{\Omega}\left|K(z,\zeta)\right|^{s}\delta_{\Omega}(\zeta)^{\varepsilon\frac{p-1}{p}q}\delta_{\Omega}(z)^{-\varepsilon\frac{p-1}{p}q}\left|f\right|^{p}(\zeta)d\mu(\zeta)\right)\\
\left(\int_{\Omega}\left|f(\zeta)\right|^{p}d\mu(\zeta)\right)^{q\frac{s-1}{s}}.
\end{multline*}
Integration with respect to the measure $d\nu(z)$ gives (using the
second hypothesis of the lemma with $\varepsilon\frac{p-1}{p}q\leq\varepsilon_{0}$)
\[
\int_{\Omega}\left|Tf(z)\right|^{q}d\nu(z)\leq C^{2}\left(\int_{\Omega}\left|f\right|^{p}d\mu\right)^{\nicefrac{q}{p}}.
\]

\end{proof}
Applying this lemma to the operator $T_{K}$ (formula (\ref{eq:operator_K1}))
with the measures $\mu=\delta_{\Omega}^{\gamma}d\lambda$ and $\nu=\delta_{\Omega}^{\gamma'}d\lambda$,
the required estimates on $K_{N}^{1}$ are summarized in the following
Lemma:
\begin{sllem}
\label{lem:estimates-kernel}Let $\mu_{0}=\frac{1-m(\gamma-\gamma')}{m(\gamma+n+1-r)+2r-1}$.\end{sllem}

\begin{enumerate}
\item For $N$ such that $-1<\gamma'\leq\gamma<N-1$ and $\varepsilon>0$
sufficiently small,
\[
\int_{\Omega}\left|K_{N}^{1}\left(z,\zeta\right)\right|^{1+\mu_{0}}\delta_{\Omega}(\zeta)^{-\mu_{0}\gamma-\varepsilon}d\lambda(\zeta)\lesssim\delta_{\Omega}(z)^{-\varepsilon}.
\]

\begin{enumerate}
\item For $-1<\gamma'$ and $N$ such that $\max\left\{ -1,\gamma-\nicefrac{1}{m}\right\} \leq\gamma'<\gamma<N-1$
and $\varepsilon>0$ sufficiently small,
\[
\int_{\Omega}\left|K_{N}^{1}\left(z,\zeta\right)\right|^{1+\mu_{0}}\frac{\delta_{\Omega}(z)^{\gamma'-\varepsilon}}{\delta_{\Omega}(\zeta)^{(1+\mu_{0})\gamma}}d\lambda(z)\lesssim\delta_{\Omega}(\zeta)^{-\varepsilon}.
\]

\end{enumerate}
\end{enumerate}

We now prove this last lemma.
\begin{proof}[Proof of (1) of \lemref{estimates-kernel}]
$K_{N}^{1}$ being bounded, uniformly in $\left(z,\zeta\right)$,
for $\zeta$ outside $P_{\varepsilon_{0}}(z)$, it is enough to prove
that
\[
\int_{P_{\varepsilon_{0}}(z)}\left|K_{N}^{1}(z,\zeta)\right|^{1+\mu_{0}}\delta_{\Omega}^{-\gamma\mu_{0}-\varepsilon}(\zeta)d\lambda(\zeta)\lesssim\delta_{\Omega}^{-\varepsilon}(z)
\]
for $\varepsilon$ sufficiently small. As this is trivial if $z$
is far from the boundary, we assume that $z$ is sufficiently close
to $\partial\Omega$.

Let $A(z,\zeta)=K_{N}^{1}(z,\zeta)\left|z-\zeta\right|^{2r-1}$. If
$\zeta\in P_{\delta_{\Omega}(z)}\left(z\right)$ then $\delta_{\Omega}(z)\simeq\delta_{\Omega}(\zeta)$
and, by (2) of \lemref{3.8-basic-estimates-kernel},
\begin{equation}
\left|A(z,\zeta)\right|^{1+\mu_{0}}\delta_{\Omega}(\zeta)^{-\gamma\mu_{0}-\varepsilon}\lesssim\frac{\delta_{\Omega}(z)^{-\mu_{0}(\gamma+n+1-r)-\varepsilon}}{\prod_{j=1}^{n-r}\tau_{j}^{2}\left(z,\delta_{\Omega}(z)\right)}.\label{eq:estimate-Z(z,zeta)-around-z}
\end{equation}
Thus, by (\ref{eq:integral-mod(z-zeta)-inverse-power-1+mu}), we get
\begin{eqnarray*}
\int_{P_{\delta_{\Omega}(z)}(z)}\left|K_{N}^{1}(z,\zeta)\right|^{1+\mu_{0}}\delta_{\Omega}(\zeta)^{-\gamma\mu_{0}-\varepsilon}d\lambda(\zeta) & \lesssim & \delta_{\Omega}(z)^{-\mu_{0}(\gamma+n+1-r+\nicefrac{(2r-1)}{m})-\varepsilon+\frac{1}{m}}\\
 & = & \delta_{\Omega}(z)^{-\varepsilon+(\gamma-\gamma')}\leq\delta_{\Omega}(z)^{-\varepsilon}.
\end{eqnarray*}

Now, let $\zeta\in P_{i}(z)=P_{2^{i}\delta_{\Omega}(z)}(z)\setminus P_{2^{(i-1)}\delta_{\Omega}(z)}(z)$,
if $N$ is sufficiently large ($N\geq\gamma+n+1$), by (2) of \lemref{3.8-basic-estimates-kernel},
we have
\[
\left|A(z,\zeta)\right|^{1+\mu_{0}}\delta_{\Omega}(\zeta)^{-\gamma\mu_{0}-\varepsilon}\lesssim\frac{\left(2^{i}\delta_{\Omega}(z)\right)^{-\mu_{0}(\gamma+n)-\varepsilon}}{\prod_{j=1}^{n-r}\tau_{j}^{2}\left(z,2^{i}\delta_{\Omega}(z)\right)}
\]

which gives, by (\ref{eq:integral-mod(z-zeta)-inverse-power-1+mu}),
\begin{eqnarray*}
\int_{P^{i}(z)}\left|K_{N}^{1}(z,\zeta)\right|^{1+\mu_{0}}\delta_{\Omega}(\zeta)^{-\gamma\mu_{0}-\varepsilon}d\lambda(\zeta) & \lesssim & \left(2^{i}\delta_{\Omega}(z)\right)^{-\mu_{0}(\gamma+n+1-r)-\varepsilon+\frac{1-(2r-1)\mu_{0}}{m}}\\
 & \leq & \delta_{\Omega}(z)^{-\varepsilon}\left(2^{i}\right)^{-\varepsilon},
\end{eqnarray*}

finishing the proof.
\end{proof}

\begin{proof}[Proof of (2) of \lemref{estimates-kernel}]
As in the preceding proof we have to show that
\[
\int_{P_{\varepsilon_{0}}(\zeta)}\left|\frac{K_{N}^{1}(z,\zeta)}{\delta_{\Omega}(\zeta)^{\gamma}}\right|^{1+\mu_{0}}\delta_{\Omega}(z)^{\gamma'-\varepsilon}d\lambda(z)\lesssim\delta_{\Omega}(\zeta)^{-\varepsilon}.
\]

If $z\in P_{\delta_{\Omega}(\zeta)}\left(\zeta\right)$ then $\delta_{\Omega}(\zeta)\simeq\delta_{\Omega}(z)$,
the estimate (\ref{eq:estimate-Z(z,zeta)-around-z}), which is still
valid replacing $\tau_{j}\left(z,\delta_{\Omega}(z)\right)$ by $\tau_{j}\left(\zeta,\delta_{\Omega}(\zeta)\right)$
(\lemref{3.4-maj-deriv-rho-equiv-tho-i-z-zeta} and (\ref{eq:integral-mod(z-zeta)-inverse-power-1+mu})),
we immediately get
\begin{eqnarray*}
\int_{P_{\delta_{\Omega}(\zeta)}\left(\zeta\right)}\left|\frac{K_{N}^{1}(z,\zeta)}{\delta_{\Omega}(\zeta)^{\gamma}}\right|^{1+\mu_{0}}\delta_{\Omega}(z)^{\gamma'-\varepsilon}d\lambda(z) & \lesssim & \delta_{\Omega}(\zeta)^{-\mu_{0}(\gamma+n+1-r)-(\gamma-\gamma')+\frac{1-(2r-1)\mu_{0}}{m}-\varepsilon}\\
 & = & \delta_{\Omega}(\zeta)^{-\mu_{0}(\gamma+n+1-r+\nicefrac{(2r-1)}{m})+\frac{1}{m}-\gamma-\gamma'-\varepsilon}\\
 & = & \delta_{\Omega}(\zeta)^{-\varepsilon}.
\end{eqnarray*}

Assume now $z\in P_{i}(\zeta)=P_{2^{i}\delta_{\Omega}(\zeta)}(\zeta)\setminus P_{2^{(i-1)}\delta_{\Omega}(\zeta)}(\zeta)$.

If $\gamma'-\varepsilon\geq0$, using $\delta_{\Omega}(z)\lesssim2^{i}\delta_{\Omega}(\zeta)$,
(2) of \lemref{3.8-basic-estimates-kernel} and (\ref{eq:integral-mod(z-zeta)-inverse-power-1+mu})
give
\begin{eqnarray*}
\int_{P_{i}(\zeta)}\left|\frac{K_{N}^{1}(z,\zeta)}{\delta_{\Omega}(\zeta)^{\gamma}}\right|^{1+\mu_{0}}\delta_{\Omega}(z)^{\gamma'-\varepsilon}d\lambda(z) & \lesssim & \left(2^{i}\delta_{\Omega}(\zeta)\right)^{-\mu_{0}(\gamma+n+1-r)-(\gamma-\gamma')+\frac{1-(2r-1)\mu_{0}}{m}-\varepsilon}\\
 & = & \left(2^{i}\delta_{\Omega}(\zeta)\right)^{-\varepsilon},
\end{eqnarray*}

finishing the proof in that case.

If $-1<\gamma'-\varepsilon$$\leq0$, as
\begin{multline*}
\int_{P_{i}(\zeta)}\left|\frac{\delta_{\Omega}(z)}{\delta_{\Omega}(\zeta)}\right|^{\gamma'-\varepsilon}\frac{d\lambda(z)}{\left(\left|z-\zeta\right|^{2r-1}\right)^{1+\mu_{0}}}\lesssim_{\gamma'-\epsilon,\mu_{0}}\\
\tau_{n-r+1}\left(\zeta,2^{i}\delta_{\Omega}(\zeta)\right)^{1-(2r-1)\mu_{0}}\prod_{j=1}^{n-r}\tau_{j}^{2}\left(\zeta,2^{i}\delta_{\Omega}(\zeta)\right),
\end{multline*}

the proof is done as before using (2) of \lemref{3.8-basic-estimates-kernel}.
\end{proof}

The proof of (1) of \thmref{d-bar-q-gamma'-p-gamma-lip} is now complete.
\end{proof}
\medskip{}

\begin{proof}[Proof of (2) and (3) of \thmref{d-bar-q-gamma'-p-gamma-lip}]
By the Hardy-Littlewood lemma we have to prove the two following
inequalities:
\begin{itemize}
\item if $p=m(\gamma+n+1-r)+2r$, $\left|\nabla_{z}\left(\int_{\Omega}f(\zeta)\wedge K_{N}^{1}(z,\zeta)\right)\right|\lesssim\delta_{\Omega}(z)^{-1}$,
\item if $p>m(\gamma+n+1-r)+2r$, $\left|\nabla_{z}\left(\int_{\Omega}f(\zeta)\wedge K_{N}^{1}(z,\zeta)\right)\right|\lesssim\delta_{\Omega}(z)^{\alpha-1}$.
\end{itemize}
Then, using H\"older's inequality these two estimates are consequences
of the following lemma:
\begin{sllem}
Let $p\geq m(\gamma+n+1-r)+2r$, $p'$ the conjugate of $p$ (i.e.
$\frac{1}{p}+\frac{1}{p'}=1$) and let $\alpha=\frac{1}{m}\left[1-\frac{m(\gamma+n+1-r)+2r}{p}\right]$.
Then
\[
\int_{\Omega}\left|\nabla_{z}K_{N}^{1}(z,\zeta)\right|^{p'}\delta_{\Omega}(\zeta)^{-\gamma\nicefrac{p'}{p}}d\lambda(\zeta)\lesssim\delta_{\Omega}(z)^{p'(\alpha-1)}.
\]
\end{sllem}

\begin{proof}[Proof of the lemma]
Denote $p'=1+\eta$ so that $\nicefrac{p'}{p}=\eta$ and $\nicefrac{1}{p}=\frac{\eta}{1+\eta}$.
Note that $\eta=\frac{1}{p-1}<\frac{1}{2r-1}$. By the basic estimates
of $K_{N}^{1}$ (and the fact that $-\frac{\gamma p'}{p}>-1$) it
suffices to estimate the above integral when the domain of integration
is reduced to $P(z,\varepsilon_{0})$.

Assume first that $\zeta\in P_{i}(z)=P_{2^{i}\delta_{\Omega}(z)}(z)\setminus P_{2^{i-1}\delta_{\Omega}(z)}(z)$.
Then, by (3) of \lemref{3.8-basic-estimates-kernel}, we have (note
that $\nicefrac{\gamma}{p}<1$)
\[
\left|\nabla_{z}K_{N}^{1}(z,\zeta)\right|\lesssim\frac{\left|\delta_{\Omega}(\zeta)\right|^{\nicefrac{\gamma}{p}}}{\left(2^{i}\delta_{\Omega}(z)\right)^{1+\nicefrac{\gamma}{p}}}\frac{1}{\prod_{j=1}^{n-r}\tau_{j}^{2}\left(z,2^{i}\delta_{\Omega}(z)\right)}\frac{1}{\left|z-\zeta\right|^{2r-1}},
\]
and by (\ref{eq:integral-mod(z-zeta)-inverse-power-1+mu}), we get
\begin{eqnarray*}
\int_{P_{i}(z)}\left|\nabla_{z}K_{N}^{1}(z,\zeta)\right|^{p'}\delta_{\Omega}(\zeta)^{-\gamma\nicefrac{p'}{p}} & d\lambda(\zeta)\lesssim & \left(2^{i}\delta_{\Omega}(z)\right)^{-(n-r+1)\eta+\frac{1-(2r-1)\eta}{m}-p'-\gamma\eta}\\
 & = & \left(2^{i}\delta_{\Omega}(z)\right)^{p'(\alpha-1)}.
\end{eqnarray*}

Assume now that $\zeta\in P_{-i}(z)=P_{2^{-(i-1)}\delta_{\Omega}(z)}(z)\setminus P_{2^{-i}\delta_{\Omega}(z)}(z)$.
Then, by (3) of \lemref{3.8-basic-estimates-kernel} (for $k=-1$)
and the fact that $\delta_{\Omega}(z)\simeq\delta_{\Omega}(\zeta)$,
we have
\[
\left|\nabla_{z}K_{N}^{1}(z,\zeta)\right|\lesssim\frac{1}{\prod_{j=1}^{n-r}\tau_{j}^{2}\left(z,2^{-i}\delta_{\Omega}(z)\right)}\frac{1}{\left|z-\zeta\right|^{2r-1}}\frac{1}{2^{-i}\delta_{\Omega}(z)},
\]
and, as before, by (\ref{eq:integral-mod(z-zeta)-inverse-power-1+mu}),
we have
\begin{eqnarray*}
\int_{P_{-i}(z)}\left|\nabla_{z}K_{N}^{1}(z,\zeta)\right|^{p'}\delta_{\Omega}(\zeta)^{-\gamma\nicefrac{p'}{p}}d\lambda(\zeta) & \lesssim & \left(2^{-i}\delta_{\Omega}(z)\right)^{p'(\alpha-1)}\left(2^{-i}\right)^{p'}\\
 & = & \left(\delta_{\Omega}(z)\right)^{p'(\alpha-1)}\left(2^{-i}\right)^{p'\alpha},
\end{eqnarray*}
finishing the proof of the lemma.
\end{proof}
The proofs of (2) and (3) of \thmref{d-bar-q-gamma'-p-gamma-lip}
are complete.
\end{proof}
The proof of \thmref{d-bar-q-gamma'-p-gamma-lip} is now complete.

\bigskip{}

\subsection{\label{sec:Proof-of-thm-est-d-bar-Nev}Proof of Theorem \ref{thm:d-bar-for-Nev}}

\quad

In this proof the form $s(z,\zeta)$ is already given by the formula
(\ref{eq:basic-choice-for-s(z,zeta)}).

For condensing, we introduce a new notation. If $U=\left\{ u_{i}\right\} _{1\leq i\leq n}$
is a set of $n$ vectors in $\mathbb{C}^{n}$, for $I=\left\{ i_{1},\ldots,i_{r}\right\} \subset I_{n}=\left\{ 1,2,\ldots,n\right\} $,
$i_{j}\neq i_{k}$ for $j\neq k$, we denote by $U_{I}$ the set
\[
U_{I}=\left\{ u_{i_{1}},\ldots,u_{i_{r}}\right\} .
\]
If $I$ is ordered increasingly, we denote by $\widehat{I}$ the subset
of $I_{n}$, ordered increasingly, such that $I_{n}=I\cup\widehat{I}$.

Writing
\[
K_{N}^{1}(z,\zeta)={\displaystyle \sideset{}{^{'}}\sum_{\left|J\right|=r-1}K_{N}^{1,J}(z,\zeta)d\overline{z}^{J}},
\]
the symbol $\sideset{}{^{'}}\sum_{J}$ meaning that the summation
is taken over increasing sets of integers $J$, if $V=\left\{ v_{1},\ldots v_{n}\right\} $
is an orthonormal basis, we write
\[
K_{N}^{1,J}(z,\zeta)\wedge f(\zeta)={\displaystyle \sideset{}{^{'}}\sum_{\left|I\right|=r}}\left\langle K_{N}^{1,J}(z,\zeta);\left(V,\overline{\widehat{V_{I}}}\right)\right\rangle \left\langle f(\zeta);\overline{V_{I}}\right\rangle ,
\]
where, denoting by $\left(\zeta_{i}^{V}\right)_{i}$ the coordinate
system associated to $V$, $\left\langle K_{N}^{1,J}(z,\zeta);\left(V,\overline{\widehat{V_{I}}}\right)\right\rangle $
is defined by
\[
K_{N}^{1,J}(z,\zeta)\wedge d\overline{\zeta^{V}}_{I}=\left\langle K_{N}^{1,J}(z,\zeta);\left(V,\overline{\widehat{V_{I}}}\right)\right\rangle d\lambda(\zeta).
\]

Thus, as we are able to estimate the kernel only when writing it in
suitable extremal coordinates, to obtain the wanted estimate, we have
to write
\[
K_{N}^{1}(z,\zeta)\wedge f(\zeta)=\sideset{}{^{'}}\sum_{\left|J\right|=r-1}\left\{ {\displaystyle \sideset{}{^{'}}\sum_{\left|I\right|=r}}\left\langle K_{N}^{1,J}(z,\zeta);\left(L,\overline{\widehat{L_{I}}}\right)\right\rangle \left\langle f(\zeta);\overline{L_{I}}\right\rangle \right\} d\overline{z}^{J}
\]
with a basis $L=\left(L_{i}(\zeta)\right)_{1\leq i\leq n}=\left(L_{i}\right)_{1\leq i\leq n}$, extremal at $\zeta$,
and apply H\"older's inequality. But this cannot be done directly since
it is well known that it is not possible to choose continuously the
bases $\left(L_{i}(\zeta)\right)_{i}$ and there is no guaranty that
the functions 
\[
\left\langle K_{N}^{1,J}(z,\zeta);\left(L,\overline{\widehat{L_{I}}}\right)\right\rangle \left\langle f(\zeta);\overline{L_{I}}\right\rangle 
\]
are measurable. To circumvent this difficulty we will choose bases
$\left(L_{i}(\zeta)\right)_{i}$ which are locally constant (thus
\emph{not extremal} at $\zeta$) and suitably close to an extremal
basis at $\zeta$. Then the properties of the geometry and the estimates
of the kernel will allow us to conclude.

\medskip{}

Let us consider a minimal covering of $\Omega$ by polydisks $P_{\delta_{\Omega}(z)}(z)$,
so that it is locally finite. Let $F$ be the union
of the boundaries of these polydisks ($F$ is closed in $\Omega$)
and let us denote by $\left\{ \Omega_{i}\right\} _{i}$ the countable
family of the connected components of $\Omega\setminus F$. For each
$i\in\mathbb{N}$ we fix an arbitrary point $Z_{i}$ in $\Omega_{i}$
and we define the coronas $R_{k}^{i}$ by
\[
R_{0}^{i}=\ring{\wideparen{P_{\delta_{\Omega}\left(Z_{i}\right)}\left(Z_{i}\right)}}\mbox{ and }R_{k}^{i}=\ring{\wideparen{P_{2^{k}\delta_{\Omega}\left(Z_{i}\right)}\left(Z_{i}\right)}}\setminus P_{2^{k-1}\delta_{\Omega}\left(Z_{i}\right)}\left(Z_{i}\right),\,k\geq1.
\]
Finally we denote $\Gamma_{i}$ the union of the boundaries of the
polydisks $R_{k}^{i}$ and $\Gamma=\cup_{i}\Gamma_{i}$.

Note that both sets $F$ and $\Gamma$ are measurable of Lebesgue
measure zero.

We now define our bases $\left(L_{j}(\zeta)\right)_{j}$ for $\left(z,\zeta\right)\in\left(\Omega\setminus F\right)\times\left(\Omega\setminus\Gamma\right)$
as follows:

\emph{let $z\in\Omega\setminus F$; then there exists a unique $i$
such that $z\in\Omega_{i}$ and a unique $k$ such that $\zeta\in R_{k}^{i}$;
then we define the vector field $L_{j}(\zeta)=L_{z,j}(\zeta)$ to
be the $j^{\mbox{th}}$ vector of a $2^{k}\delta_{\Omega}\left(Z_{i}\right)$-extremal
basis at the point $Z_{i}$ used for the polydisk $P_{2^{k}\delta_{\Omega}\left(Z_{i}\right)}\left(Z_{i}\right)$}
and set 
\[
L(\zeta)=\left(L_{j}(\zeta)\right)_{1\leq j\leq n}.
\]

With this definition the expressions 
\[
{\displaystyle \sideset{}{^{'}}\sum_{\left|I\right|=r}}\left\langle K_{N}^{1,J}(z,\zeta);\left(L(\zeta),\overline{\widehat{L_{I}(\zeta)}}\right)\right\rangle \left\langle f(\zeta),\overline{L_{I}(\zeta)}\right\rangle 
\]
is a measurable function on $\Omega\times\Omega$ (because it is the
restriction to a set of total measure of a locally smooth function)
and we can write, for $z\in\Omega\setminus F$,
\begin{eqnarray*}
u(z) & := & \int_{\Omega}K_{N}^{1}(z,\zeta)\wedge f(\zeta)\\
 & = & {\displaystyle \sideset{}{^{'}}\sum_{\left|J\right|=r-1}\left\{ \sideset{}{^{'}}\sum_{\left|I\right|=r}\int_{\Omega\setminus\Gamma}\left\langle K_{N}^{1,J}(z,\zeta);\left(L(\zeta),\overline{\widehat{L_{I}(\zeta)}}\right)\right\rangle \left\langle f(\zeta),\overline{L_{I}(\zeta)}\right\rangle d\lambda(\zeta)\right\} d\overline{z}^{J}}.
\end{eqnarray*}

The kernel $K_{N}^{1}(z,\zeta)$ being uniformly integrable in the
variable $\zeta$ (c.f. \cite{CDMb}), H\"older's inequality gives
\[
\left|u(z)\right|^{p}\lesssim{\displaystyle \sideset{}{^{'}}\sum_{\left|J\right|=r-1}\left\{ \sideset{}{^{'}}\sum_{\left|I\right|=r}\int_{\Omega\setminus\Gamma}\left|\left\langle K_{N}^{1,J}(z,\zeta);\left(L(\zeta),\overline{\widehat{L_{I}(\zeta)}}\right)\right\rangle \right|\left|\left\langle f(\zeta),\overline{L_{I}(\zeta)}\right\rangle \right|^{p}d\lambda(\zeta)\right\} }.
\]

The measurability of the functions
\[
\left(z,\zeta\right)\mapsto\left|\left\langle K_{N}^{1,J}(z,\zeta);\left(L(\zeta),\overline{\widehat{L_{I}(\zeta)}}\right)\right\rangle \right|\left|\left\langle f(\zeta),\overline{L_{I}(\zeta)}\right\rangle \right|^{p},
\]
the facts that $F$ and $\Gamma$ are of Lebesgue measure zero and
Fubini's theorem finally give
\begin{multline*}
\int_{\Omega}\left|u(z)\right|^{p}\delta_{\Omega}^{\alpha-1}(z)d\lambda(z)\lesssim{\displaystyle \sideset{}{^{'}}\sum_{\left|J\right|=r-1}\left\{ \int_{\Omega\setminus\Gamma}\left[\sideset{}{^{'}}\sum_{\left|I\right|=r}\int_{\Omega\setminus F}\left|\left\langle K_{N}^{1,J}(z,\zeta);\left(L(\zeta),\overline{\widehat{L_{I}(\zeta)}}\right)\right\rangle \right|\right.\right.}\\
\left.\left.\left|\left\langle f(\zeta),\overline{L_{I}(\zeta)}\right\rangle \right|^{p}\delta_{\Omega}^{\alpha-1}(z)d\lambda(z)\right]d\lambda(\zeta)\right\} .
\end{multline*}

For $\zeta\in\Omega\setminus\Gamma$ let us consider the coronas $Q_{0}(\zeta)=P_{\delta_{\Omega}(\zeta)}(\zeta)$
and $Q_{t}(\zeta)=P_{2^{t}\delta_{\Omega}(\zeta)}(\zeta)\setminus P_{2^{t-1}\delta_{\Omega}(\zeta)}(\zeta)$,
$t\geq1$. We now evaluate the integrals
\[
\int_{\left(\Omega\setminus\Gamma\right)\cap Q_{t}(\zeta)}\left|\left\langle K_{N}^{1,J}(z,\zeta):\left(L(\zeta),\overline{\widehat{L_{I}(\zeta)}}\right)\right\rangle \right|\left|\left\langle f(\zeta),\overline{L_{I}(\zeta)}\right\rangle \right|^{p}\delta_{\Omega}^{\alpha-1}(z)d\lambda(z).
\]

Recall that if $z\in\Omega_{s}$ there exists a unique $k$ such that
$\zeta\in R_{k}^{s}$ and $L_{j}(\zeta)$ is the $j^{\mbox{th}}$
vector of the $2^{k}\delta_{\Omega}\left(Z_{s}\right)$-extremal basis
at $Z_{s}$ chosen before. To simplify the notations, we denote by $\left(\zeta_{l}\right)_{l}$
the coordinate system associated to that basis so that $L_{l}(\zeta)=\frac{\partial}{\partial\zeta_{l}}$.
Writing everything in those coordinate systems, we have to integrate
$\left|K_{N}^{1,J}(z,\zeta)\wedge d\overline{\zeta}_{I}\right|\left|f_{I}(\zeta)\right|^{p}$
over $\left(\Omega\setminus\Gamma\right)\cap Q_{t}(\zeta)$ where
$f_{I}(\zeta)=\left\langle f(\zeta),\overline{L_{I}(\zeta)}\right\rangle $.

First, we remark that $K_{N}^{1,J}(z,\zeta)\wedge d\overline{\zeta}_{I}$
is a sum of expressions of the form $\frac{W}{D}$ where
\[
D(\zeta,z)=\left|z-\zeta\right|^{2r}\left(\frac{1}{K_{0}}S(z,\zeta)+\rho(\zeta)\right)^{n+N-r},
\]
and,
\[
W=\left(\overline{\zeta}_{m}-\overline{z}_{m}\right)\rho(\zeta)^{N}\prod_{k=1}^{n-r}\frac{\partial Q_{i_{k}}(z,\zeta)}{\partial\overline{\zeta}_{j_{k}}}\bigwedge_{i=1}^{n}\left(d\zeta_{i}\wedge d\overline{\zeta_{i}}\right)
\]
or
\[
W=\left(\overline{\zeta}_{m}-\overline{z}_{m}\right)\rho(\zeta)^{N-1}\frac{\partial\rho(\zeta)}{\partial\overline{\zeta}_{j_{k_{0}}}}Q_{i_{k_{0}}}(\zeta,z)\prod_{\substack{1\leq k\leq n-r\\ k\neq k_{0}}}\frac{\partial Q_{i_{k}}(z,\zeta)}{\partial\overline{\zeta}_{j_{k}}}\bigwedge_{i=1}^{n}\left(d\zeta_{i}\wedge d\overline{\zeta_{i}}\right),
\]
with $\left\{ j_{1},\ldots,j_{n-r},I\right\} =\left\{ 1,\ldots,n\right\} $,
and, $i_{1},\ldots,i_{n-r},m$ all different.

Then, using \lemref{3.6-Lemma-maj-deriv_rho_Q} and the properties
of the geometry, we obtain the following estimates:
\begin{itemize}
\item For $z\in\Omega_{s}\cap Q_{0}(\zeta)$, using inequality (\ref{eq:real-part-Q-zeta-z-plus-1}),
(3) of \lemref{3.4-maj-deriv-rho-equiv-tho-i-z-zeta} and the fact
that $\delta(\zeta)\simeq\delta\left(Z_{s}\right)$, $\left|K_{N}^{1,J}(z,\zeta)\wedge d\overline{\zeta}_{I}\right|$
is a sum of expressions bounded by 
\[
\frac{1}{\prod_{j=1}^{n-r}\tau_{j}^{2}\left(Z_{s},\delta_{\Omega}(\zeta)\right)}\frac{\prod_{i\in I}\tau_{i}\left(Z_{s},\delta_{\Omega}(\zeta)\right)}{\prod_{j=n-r+1}^{n}\tau_{j}\left(Z_{s},\delta_{\Omega}(\zeta)\right)}\frac{1}{\left|z-\zeta\right|^{2r-1}},
\]
and, using (1) and (2) of \lemref{3.4-maj-deriv-rho-equiv-tho-i-z-zeta},
these expressions are bounded by
\[
\frac{1}{\prod_{j=1}^{n-r}\tau_{j}^{2}\left(\zeta,\delta_{\Omega}(\zeta)\right)}\frac{\prod_{i\in I}\tau\left(\zeta,L_{i}(\zeta),\delta_{\Omega}(\zeta)\right)}{\prod_{j=n-r+1}^{n}\tau_{j}\left(\zeta,\delta_{\Omega}(\zeta)\right)}\frac{1}{\left|z-\zeta\right|^{2r-1}},
\]
thus
\begin{gather}
\left|K_{N}^{1,J}(z,\zeta)\wedge d\overline{\zeta}_{I}\right|\left|f_{I}(\zeta)\right|^{p}\lesssim\nonumber \\
\frac{1}{\prod_{j=1}^{n-r}\tau_{j}^{2}\left(\zeta,\delta_{\Omega}(\zeta)\right)}\frac{1}{\tau_{n+r-1}\left(\zeta,\delta_{\Omega}(\zeta)\right)}\frac{1}{\left|z-\zeta\right|^{2r-1}}\delta_{\Omega}(\zeta)\left\Vert f(\zeta)\right\Vert _{\mathbbmss k,p}.\label{eq:maj-K-I-f-I-anisotrope}
\end{gather}
This gives (using (\ref{eq:integral-mod(z-zeta)-incerse-delta-alpha-1}))
\[
\int_{\left(\Omega\setminus F\right)\cap Q_{0}(\zeta)}\left|K_{N}^{1,J}(z,\zeta)\wedge d\overline{\zeta}_{I}\right|\left|f_{I}(\zeta)\right|^{p}\delta_{\Omega}^{\alpha-1}(z)d\lambda(z)\lesssim\frac{\delta_{\Omega}^{\alpha}(\zeta)}{\min\left\{ \alpha,1\right\} }\left\Vert f(\zeta)\right\Vert _{\mathbbmss k,p}.
\]

\item Similarly, for $z\in\Omega_{s}\cap Q_{t}(\zeta)$, $t\geq1$, then
$\zeta\in R_{k}^{s}$ with $2^{t}\delta_{\Omega}(\zeta)\simeq2^{k}\delta_{\Omega}\left(Z_{s}\right)$,
and, using \lemref{3.5-lemma-maj-rho-S} instead of inequality (\ref{eq:real-part-Q-zeta-z-plus-1})
(as in \lemref{3.8-basic-estimates-kernel} (2)) and (\ref{eq:comp-tau-epsilon-tau-lambda-epsilon}),
$\left|K_{N}^{1,J}(z,\zeta)\wedge d\overline{\zeta}_{I}\right|\left|f_{I}(\zeta)\right|^{p}$
is a sum of expressions bounded by
\begin{gather*}
\left(\frac{\delta_{\Omega}(\zeta)}{2^{t}\delta_{\Omega}(\zeta)}\right)^{N-1}\frac{1}{\prod_{j=1}^{n-r}\tau_{j}^{2}\left(\zeta,2^{t}\delta_{\Omega}(\zeta)\right)}\frac{1}{\tau_{n+r-1}\left(\zeta,2^{t}\delta_{\Omega}(\zeta)\right)}\times\\
\frac{2^{tr}}{\left|z-\zeta\right|^{2r-1}}\delta_{\Omega}(\zeta)\left\Vert f(\zeta)\right\Vert _{\mathbbmss k,p},
\end{gather*}
giving, for $N\geq\alpha+r+3$,
\[
\int_{\left(\Omega\setminus F\right)\cap Q_{t}(\zeta)}\left|K_{N}^{1,J}(z,\zeta)\wedge d\overline{\zeta}_{I}\right|\left|f_{I}(\zeta)\right|^{p}\delta_{\Omega}^{\alpha-1}(z)d\lambda(z)\lesssim\frac{\delta_{\Omega}^{\alpha}(\zeta)}{\min\left\{ \alpha,1\right\} }\frac{\left\Vert f(\zeta)\right\Vert _{\mathbbmss k,p}}{2^{t}}.
\]
Indeed, if $\alpha\geq1$, $z\in Q_{t}(\zeta)$ implies $\delta_{\Omega}(z)\lesssim2^{t}\delta_{\Omega}(\zeta)$
and if $0<\alpha<1$, using the proof of \lemref{Lemma-3-9}.
\end{itemize}

This concludes the proof of \thmref{d-bar-for-Nev}.

\subsection{Proof of Theorem \ref{thm:d-bar-boundary-0-q-forms}}

\quad

As $f$ is assumed to be smooth in $\overline{\Omega}$, \thmref{d-bar-q-gamma'-p-gamma-lip}
implies that the form $z\mapsto u(z)=\int_{\Omega}K_{N}^{1}(z,\zeta)\wedge f(\zeta)$
is continuous on $\overline{\Omega}$.

The proof starts, as in the previous section, considering the sets
$F$ and $\Gamma$ and the vectors $\left(L_{j}(\zeta)\right)_{j}$.
Note that, as $F$ is of Lebesgue measure zero, for almost all $\varepsilon>0$,
$\varepsilon\leq\varepsilon_{0}$, $\varepsilon_{0}$ small, $\sigma_{\varepsilon}\left(\left\{ \rho=-\varepsilon\right\} \cap F\right)=0$,
$\sigma_{\varepsilon}$ being the euclidean measure on $\left\{ \rho=-\varepsilon\right\} $.
Then, the proof is done showing that there exists a constant $C>0$
such that, for such $\varepsilon$,
\[
\int_{\left\{ \rho=-\varepsilon\right\} \setminus F}\left|u\right|^{p}d\sigma_{\varepsilon}\leq C\int_{\Omega}\left\Vert f\right\Vert _{\mathbbmss k,p}d\lambda.
\]

We do it using almost the same proof as in the previous section so
we will not give details. Simply, note that using the same notations
for $\Omega_{s}$, $Q_{i}(\zeta)$ and writing $K_{N}^{1,J}(z,\zeta)\wedge f(\zeta)$
in the same extremal bases, (\ref{eq:maj-K-I-f-I-anisotrope}) implies
\[
\int_{\left(\left\{ \rho=-\varepsilon\right\} \setminus F\right)\cap Q_{0}(\zeta)}\left|K_{N}^{1,J}(z,\zeta)\wedge d\overline{\zeta_{I}}\right|\left|f_{I}(\zeta)\right|^{p}d\sigma(z)\lesssim\left\Vert f(\zeta)\right\Vert _{\mathbbmss k,p}
\]
and
\[
\int_{\left(\left\{ \rho=-\varepsilon\right\} \setminus F\right)\cap Q_{i}(\zeta)}\left|K_{N}^{1,J}(z,\zeta)\wedge d\overline{\zeta_{I}}\right|\left|f_{I}(\zeta)\right|^{p}d\sigma(z)\lesssim\frac{1}{2^{i}}\left\Vert f(\zeta)\right\Vert _{\mathbbmss k,p},
\]
for $N\geq5+r$. Finally we obtain
\[
\int_{\left\{ \rho=-\varepsilon\right\} }\left|u\right|^{p}d\sigma\lesssim\int_{\Omega}\left\Vert f(\zeta)\right\Vert _{\mathbbmss k,p}d\lambda(\zeta).
\]

\subsection{Proof of Theorem \ref{thm:Carleson-estimates-for-d-bar}}

\quad

Our proof is very similar to the one given by N. Nguyen in \cite{Ngu01}
for convex domains. As this paper was not published, we give some
details below.

A standard interpolation argument shows that we only have to prove
the BMO estimate. The form $s$ used in the previous proofs (formula
(\ref{eq:basic-choice-for-s(z,zeta)})) is not well adapted to a BMO-estimate,
so, we change it here and make a choice similar to the one made in
\cite{DM01}:
\[
s(z,\zeta)=-\rho(z)\sum\left(\overline{\zeta_{i}}-\overline{z_{i}}\right)d\left(\zeta_{i}-z_{i}\right)+\chi(z)\overline{S(\zeta,z)}\sum Q_{i}\left(\zeta,z\right)d\left(\zeta_{i}-z_{i}\right),
\]
where $S$ and $Q_{i}$ are given by \eqref{def-S-and-Qi} and $\chi$
is a smooth cut-off function equal to $0$ when $\rho(z)<-\eta_{0}$
and to $1$ when $\rho(z)>-\nicefrac{\eta_{0}}{2}$, $\eta_{0}$ sufficiently
small. Clearly, for $z$ in a compact subset of $\Omega$ and $\zeta$
in $\Omega$, 
\[
\left|\left\langle s(z,\zeta),z-\zeta\right\rangle \right|\gtrsim\left|\zeta-z\right|^{2}
\]
and, on $\Omega\times\Omega$, 
\[
\left|\left\langle s(z,\zeta),z-\zeta\right\rangle \right|\lesssim\left|\zeta-z\right|
\]
and, as $S$ and $Q_{i}$ are unchanged, formula (\ref{eq:And-Ber-formula})
is still valid. Moreover \lemref{3-3-C1-regularity-PN} remains unchanged
and the proof of the theorem is reduced to the proof of the same estimate
for the operator $T_{K}$ (formula (\ref{eq:operator_K1})).

To get the continuity up to the boundary of $z\mapsto\int_{\Omega}K_{N}^{1}(z,\zeta)\wedge f(\zeta)$
we need to prove the uniform integrability of our new kernel $K_{N}^{1}$
that is the following lemma:
\begin{sllem}
For $z$ sufficiently close to $\partial\Omega$ and $\varepsilon$
sufficiently small, if $K_{N}$ is the component of $K_{N}^{1}$ of
bi-degree $\left(n,n-1\right)$ in $\zeta$ and $\left(0,0\right)$
in $z$, 
\[
\int_{P_{\varepsilon}(z)}\left|K_{N}(z,\zeta)\right|d\lambda(\zeta)\lesssim\varepsilon^{\frac{1}{m+1}}.
\]
\end{sllem}

\begin{proof}
Under the conditions of the lemma, the kernel $K_{N}$ is reduced
to
\[
\left(\frac{\rho(\zeta)}{\rho(\zeta)+\frac{1}{K_{0}}S(z,\zeta)}\right)^{N+n-1}\frac{\left(\partial_{\overline{\zeta}}Q\right)^{n-1}\sum\left(\overline{S(\zeta,z)}Q_{i}(\zeta,z)-\rho(z)\left(\overline{\zeta_{i}}-\overline{z_{i}}\right)\right)d\zeta_{i}}{-\rho(z)\left|\zeta-z\right|^{2}+\left|S(\zeta,z)\right|^{2}}.
\]

Let us denote $P_{l}=P_{2^{-l}\varepsilon}(z)\setminus P_{2^{-l-1}\varepsilon}(z)$
and $\frac{W}{D}=\frac{\overline{S(\zeta,z)}Q_{i}(\zeta,z)-\rho(z)\left(\overline{\zeta_{i}}-\overline{z_{i}}\right)}{-\rho(z)\left|\zeta-z\right|^{2}+\left|S(\zeta,z)\right|^{2}}$
expressed in the $\left(z,2^{-l}\varepsilon\right)$-extremal basis.

Let us estimate $\frac{W}{D}$ for $\zeta\in P_{l}$. 

Clearly $\frac{\left|-\rho(z)\left(\overline{\zeta_{i}}-\overline{z_{i}}\right)\right|}{D}\lesssim\frac{1}{\left|\zeta-z\right|}$.

For the first term of $W$, suppose $\left|S(\zeta,z)\right|\gtrsim2^{-l}\varepsilon$.
As $\left|Q_{i}(\zeta,z)\right|\lesssim\frac{2^{-l}\varepsilon}{\tau_{i}\left(z,2^{-l}\varepsilon\right)}$
(\lemref{3.6-Lemma-maj-deriv_rho_Q}), we have
\[
\frac{\left|\overline{S(\zeta,z)}Q_{i}(\zeta,z)\right|}{D}\lesssim\frac{1}{\tau_{i}\left(z,2^{-l}\varepsilon\right)}.
\]
Conversely, if $\left|S(\zeta,z)\right|\ll2^{-l}\varepsilon$, by
\lemref{3.5-lemma-maj-rho-S}, $\rho(z)\gtrsim2^{-l}\varepsilon$
and
\[
\frac{\left|\overline{S(\zeta,z)}Q_{i}(\zeta,z)\right|}{D}\lesssim\frac{\left(2^{-l}\varepsilon\right)^{2}}{2^{-l}\varepsilon\tau_{i}\left(z,2^{-l}\varepsilon\right)\left|\zeta-z\right|^{2}}.
\]

Then, using lemmas \ref{lem:3.4-maj-deriv-rho-equiv-tho-i-z-zeta},
\ref{lem:3.5-lemma-maj-rho-S} and \ref{lem:3.6-Lemma-maj-deriv_rho_Q},
on $P_{l}$, $\left|K_{N}\right|$ is bounded by a sum of expressions
of the type
\[
\frac{\tau_{i}\left(z,2^{-l}\varepsilon\right)\tau_{k}\left(z,2^{-l}\varepsilon\right)}{\prod_{j=1}^{n}\tau_{j}^{2}\left(z,2^{-l}\varepsilon\right)}\left[\frac{1}{\tau_{i}\left(z,2^{-l}\varepsilon\right)}+\frac{2^{-l}\varepsilon}{\tau_{i}\left(z,2^{-l}\varepsilon\right)\left|\zeta-z\right|^{2}}+\frac{1}{\left|\zeta-z\right|}\right].
\]

Now, by \lemref{Lemma-3-9}
\[
\int_{P_{2^{-l}\varepsilon}(z)}\frac{d\lambda}{\left|\zeta-z\right|}\lesssim\frac{\prod_{j=1}^{n}\tau_{j}^{2}\left(z,2^{-l}\varepsilon\right)}{\tau_{n}\left(z,2^{-l}\varepsilon\right)},
\]
and, if $n\geq3,$
\begin{eqnarray*}
\int_{P_{2^{-l}\varepsilon}(z)}\frac{d\lambda}{\left|\zeta-z\right|^{2}} & \lesssim & \tau_{n}\left(z,2^{-l}\varepsilon\right)\int_{P_{2^{-l}\varepsilon}(z)}\frac{d\lambda}{\left|\zeta-z\right|^{3}}\\
 & \lesssim & \prod_{j=1}^{n-2}\tau_{j}^{2}\left(z,2^{-l}\varepsilon\right)\tau_{n-1}\left(z,2^{-l}\varepsilon\right)\tau_{n}\left(z,2^{-l}\varepsilon\right),
\end{eqnarray*}
and, if $n=2$,
\[
\int_{P_{2^{-l}\varepsilon}(z)}\frac{d\lambda}{\left|\zeta-z\right|^{2}}\lesssim\tau_{1}\left(z,2^{-l}\varepsilon\right)^{2}\log\left(\frac{1}{\tau_{2}\left(z,2^{-l}\varepsilon\right)}\right).
\]

This finishes the proof of the lemma.
\end{proof}
\medskip{}

We now follow the (unpublished) method developed by N. Nguyen in \cite{Ngu01}
to prove the BMO-estimate for the function $u(z)=\int_{\Omega}K_{N}(z,\zeta)\wedge f(\zeta)$,
$z\in\partial\Omega$: for $z_{0}\in\partial\Omega$ and $\varepsilon$
small, denoting $B_{0}=P_{\varepsilon}(z_{0})\cap\partial\Omega$,
we have to estimate
\[
\frac{1}{\sigma\left(B_{0}\right)}\int_{B_{0}}\left|u(z)-u_{B_{0}}\right|d(z)
\]
where $\sigma$ is the euclidean measure on $\partial\Omega$ and
$u_{B_{0}}=\frac{1}{\sigma\left(B_{0}\right)}\int_{B_{0}}u(w)d\sigma(w)$.

Let $\Omega_{1}=\Omega\cap P_{C\varepsilon}\left(z_{0}\right)$ and
$\Omega_{2}=\Omega\setminus\Omega_{1}$ where $C>0$ is a sufficiently
large number, independent of $\varepsilon$ and $z_{0}$, that will
be fixed later. Then
\[
\sigma\left(B_{0}\right)\int_{B_{0}}\left|u(z)-u_{B_{0}}\right|d\sigma(z)\leq I_{1}+I_{2}
\]
where, for $i=1,2$
\[
I_{i}=\int_{B_{0}}\int_{B_{0}}\int_{\Omega_{i}}\left|\left(K_{N}(z,\zeta)-K_{N}(w,\zeta)\right)\wedge f(\zeta)\right|d\lambda(\zeta)d\sigma(z)d\sigma(w).
\]

To estimate $I_{1}$ we simply write
\[
I_{1}\leq2\sigma\left(B_{0}\right)\int_{B_{0}}\int_{\Omega_{1}}\left|K_{N}(z,\zeta)\wedge f(\zeta)\right|d\lambda(\zeta)d\sigma(z).
\]

For $\zeta\in\Omega_{1}$ let $B_{i}(\zeta)=B_{0}\cap\left(P_{2^{i}\delta_{\Omega}(\zeta)}(\zeta)\setminus P_{2^{i-1}\delta_{\Omega}(\zeta)}(\zeta)\right)$,
$i\geq1$ (recall that, by the choice of $c_{0}$ in the definition
(\ref{eq:def-polydisk}) of the polydisks, $P_{\delta_{\Omega}(\zeta)}(\zeta)\cap\partial\Omega=\emptyset$).
Then, if $z\in B_{i}(\zeta)$, writing $K_{N}(z,\zeta)\wedge f(\zeta)$
in the $\left(\zeta,2^{i}\delta_{\Omega}(\zeta)\right)$-extremal
basis and using lemmas \ref{lem:3.4-maj-deriv-rho-equiv-tho-i-z-zeta},
\ref{lem:3.5-lemma-maj-rho-S} and \ref{lem:3.6-Lemma-maj-deriv_rho_Q},
$\left|Q_{j}(\zeta,z)\right|$ is bounded by $\frac{2^{i}\delta_{\Omega}(\zeta)}{\tau_{j}\left(\zeta,2^{i}\delta_{\Omega}(\zeta)\right)}$
and $\left|S(\zeta,z)\right|$ is bounded from below by $2^{i}\delta_{\Omega}(\zeta)$
(because $z\in\partial\Omega$), and we get
\begin{eqnarray*}
\left|K_{N}(z,\zeta)\wedge f_{l}d\zeta_{l}\right| & \lesssim & \left|\frac{\delta_{\Omega}(\zeta)}{2^{i}\delta_{\Omega}(\zeta)}\right|^{N-1}\frac{2^{i}\delta_{\Omega}(\zeta)}{{\displaystyle \prod_{j=1}^{n}}\tau_{j}^{2}\left(\zeta,2^{i}\delta_{\Omega}(\zeta)\right)}\frac{\tau\left(\zeta,L_{l}(\zeta),2^{i}\delta_{\Omega}(\zeta)\right)}{2^{i}\delta_{\Omega}(\zeta)}\left|f_{l}(\zeta)\right|\\
 & \lesssim & \left(\frac{1}{2^{i}}\right)^{N-1}\frac{1}{\sigma\left(B_{i}(\zeta)\right)}\left\Vert f(\zeta)\right\Vert _{\mathbbmss k},
\end{eqnarray*}
the last inequality coming from (\ref{eq:comp-tau-epsilon-tau-lambda-epsilon}).
Integrating over $B_{i}(\zeta)$ and summing up we obtain ($C$ being
fixed, $\sigma\left(B_{0}\right)\simeq\sigma\left(\Omega_{1}\cap\partial\Omega\right)$)
\[
\frac{I_{1}}{\left(\sigma\left(B_{0}\right)\right)^{2}}\lesssim\frac{1}{\sigma\left(\Omega_{1}\cap\partial\Omega\right)}\int_{\Omega_{1}}\left\Vert f(\zeta)\right\Vert _{\mathbbmss k}\leq\left\Vert \left\Vert f(\zeta)\right\Vert _{\mathbbmss k}d\lambda\right\Vert _{W^{1}}.
\]

To estimate $I_{2}$ it is necessary to evaluate the difference $\left(K_{N}(z,\zeta)-K_{N}(w,\zeta)\right)$.
When $\Omega$ is convex, in \cite{Ngu01} N. Nguyen estimates the
derivatives of $K_{N}(\cdot,\zeta)$ on the segment $\left[z,w\right]$.
Here, this segment is not necessary contained in $\overline{\Omega}$
and we introduce two points $Z$ and $W$ which are $c_{1}\varepsilon$-translations
of $z$ and $w$ in the direction of the inward real normal at $z_{0}$,
$c_{1}$ being chosen sufficiently large (independent of $z_{0}$
and $\varepsilon$) so that the segment $\left[Z,W\right]$ is contained
in $\Omega$. Thus the segments $\left[z,Z\right]$, $\left[Z,W\right]$
and $\left[W,w\right]$ are contained in a polydisk $P_{c_{2}\varepsilon}\left(z_{0}\right)$
($c_{2}$ independent of $z_{0}$ and $\varepsilon$) and $C$ is
chosen sufficiently large so that, if $\zeta\in\Omega_{2}$, for all
$u\in P_{c_{2}\varepsilon}\left(z_{0}\right)$, the anisotropic distance
from $\zeta$ to $u$ is equivalent to the distance from $\zeta$to
$z_{0}$.

The points $z$ and $w$ being on the boundary of $\Omega$, we have
to control the variations of
\[
\left(\frac{\rho(\zeta)}{\frac{1}{K_{0}}S(z,\zeta)+\rho(\zeta)}\right)^{n-1+N}\frac{\left(\partial_{\overline{\zeta}}Q(z,\zeta)\right)^{n-1}\sum Q_{j}(\zeta,z)d\zeta_{j}}{S(\zeta,z)}=K_{N}'(z,\zeta).
\]

Let $Q_{i}=P_{2^{i}C\varepsilon}\left(z_{0}\right)\setminus P_{2^{i-1}C\varepsilon}\left(z_{0}\right)$,
$i\geq1$, and let us estimate the derivatives of $K_{N}'(u,\zeta)$
for $u\in P_{c_{2}\varepsilon}\left(z_{0}\right)$ (in particular
in the three segments) and $\zeta\in Q_{i}$. Using lemmas \ref{lem:3.4-maj-deriv-rho-equiv-tho-i-z-zeta},
\ref{lem:3.5-lemma-maj-rho-S} and \ref{lem:3.6-Lemma-maj-deriv_rho_Q},
and the fact that, if $u\in P_{c_{2}\varepsilon}\left(z_{0}\right)$,
$\left|\rho(u)\right|\lesssim c_{2}\varepsilon$, enlarging $C$ if
necessary, $S(z,u)\gtrsim2^{i}C\varepsilon$ (by \lemref{3.5-lemma-maj-rho-S}),
writing $K_{N}'$ and $f$ in the $\left(z_{0},2^{i}C\varepsilon\right)$-extremal
basis we get
\begin{multline*}
\left|\frac{\partial}{\partial u_{t}}K_{N}'(u,\zeta)\wedge d\overline{\zeta_{l}}\right|+\left|\frac{\partial}{\partial\overline{u_{t}}}K_{N}'(u,\zeta)\wedge d\overline{\zeta_{l}}\right|\lesssim\\
\frac{2^{i}C\varepsilon}{\prod_{j=1}^{n}\tau_{j}^{2}\left(z_{0},2^{i}C\varepsilon\right)}\frac{\tau\left(z_{0},L_{l},2^{i}C\varepsilon\right)}{2^{i}C\varepsilon}\frac{1}{\tau\left(z_{0},L_{t},2^{i}C\varepsilon\right)}.
\end{multline*}

On each of the three segments $\left[z,Z\right]$, $\left[Z,W\right]$
and $\left[W,w\right]$ the increase in the direction $L_{t}$ is
bounded, up to a multiplicative constant, by $\tau\left(z_{0},L_{t},c_{2}\varepsilon\right)$,
and, by (\ref{geometry-5}) of the geometry, 
\[
\tau\left(z_{0},L_{t}2^{i}C\varepsilon\right)\gtrsim\left(\frac{2^{i}C}{c_{2}}\right)^{\nicefrac{1}{m}}\tau\left(z_{0},L_{t},c_{2}\varepsilon\right),
\]
so
\[
\left|\left(K_{N}(z,\zeta)-K_{N}(w,\zeta)\right)\wedge f_{l}(\zeta)d\overline{\zeta_{l}}\right|\lesssim\frac{\left|f_{l}(\zeta)\right|}{\sigma\left(Q_{i}\cap\partial\Omega\right)}\frac{\tau\left(\zeta,L_{l},\delta_{\Omega}(\zeta)\right)}{\delta_{\Omega}(\zeta)}\left(\frac{1}{2^{i}}\right)^{\nicefrac{1}{m}},
\]
because $\tau\left(z_{0},L_{l},2^{i}C\varepsilon\right)\simeq\tau\left(\zeta,L_{l},2^{i}C\varepsilon\right)$
(\lemref{3.4-maj-deriv-rho-equiv-tho-i-z-zeta}) and
\[
\frac{\tau\left(\zeta,L_{l},2^{i}C\varepsilon\right)}{2^{i}C\varepsilon}\lesssim\frac{\tau\left(\zeta,L_{l},\delta_{\Omega}(\zeta)\right)}{\delta_{\Omega}(\zeta)}
\]
because $\delta_{\Omega}(\zeta)\lesssim2^{I}c\varepsilon$. Finally
\[
\int_{Q_{i}}\left|\left(K_{N}(z,\zeta)-K_{N}(w,\zeta)\right)\wedge f_{l}(\zeta)d\overline{\zeta_{l}}\right|d\lambda(\zeta)\lesssim\left(\frac{1}{2^{i}}\right)^{\nicefrac{1}{m}}\frac{1}{\sigma\left(Q_{i}\cap\partial\Omega\right)}\int_{Q_{i}}\left\Vert f(\zeta)\right\Vert _{\mathbbmss k},
\]
giving $\frac{I_{2}}{\sigma\left(B_{0}\right)^{2}}\lesssim\left\Vert \left\Vert f(\zeta)\right\Vert _{\mathbbmss k}d\lambda\right\Vert _{W^{1}}$,
and the proof is complete.

\section{Proof of Theorem \ref{thm:estimates-bergman}}

We use the method developed for the proofs of theorems 2.1 and 2.3
of \cite{CPDY}.

In \cite{CDM} we proved the following result: let $g$ be a gauge
of $D$ and $\rho_{0}=g^{4}e^{1-\nicefrac{1}{g}}-1$ then:
\begin{stthm}[Theorem 2.1 of \cite{CDM}]
\label{thm:regularity-P-omega-0}Let $\omega_{0}=\left(-\rho_{0}\right)^{r}$,
$r$ being a non negative rational number, and let $P_{\omega_{0}}$
be the Bergman projection of the Hilbert space $L^{2}\left(D,\omega_{0}\right)$.
Then, for $p\in\left]1,+\infty\right[$ and $1\leq\beta\leq p\left(r+1\right)-1$,
$P_{\omega_{0}}$ maps continuously the space $L^{p}\left(D,\delta_{D}^{\beta}\right)$
into itself and, for $\alpha>0$, $P_{\omega_{0}}$ maps continuously
the lipschitz space $\Lambda_{\alpha}(D)$ into itself.
\end{stthm}

If $\omega$ is as in \thmref{estimates-bergman} then there exists
a strictly positive $\mathcal{C}^{1}$ function in $\overline{D}$,
$\varphi$, such that $\omega=\varphi\omega_{0}$. Then we compare
the regularity of $P_{\omega_{0}}$ and $P_{\omega}$ using the following
formula (Proposition 3.1 of \cite{CPDY}): for $u\in L^{2}\left(D,\omega\right)$,
\[
\varphi P_{\omega}(u)=P_{\omega_{0}}(\varphi u)+\left(\mathrm{Id}-P_{\omega_{0}}\right)\circ A\left(P_{\omega}(u)\wedge\overline{\partial}\varphi\right),
\]
where $A$ is any operator solving the $\overline{\partial}$-equation
for $\overline{\partial}$-closed forms in $L^{2}\left(D,\omega\right)$.

We first show that $P_{\omega}$ maps continuously $L^{p}\left(D,\delta_{\Omega}^{r}\right)$
into itself. Let $f\in L^{p}\left(D,\delta_{\Omega}^{r}\right)$,
$p\in\left[2,+\infty\right[$. For $A$ we choose the operator $T$
of \propref{d-bar-gain-exponent} with $\gamma=r$, and we choose
$\varepsilon$ and an integer $M$ such that $0<\varepsilon\leq\varepsilon_{N}$,
$\varepsilon_{N}$ as in \propref{d-bar-gain-exponent}, and $p=2+M\varepsilon$.
Let us prove, by induction, that $P_{\omega}(f)\in L^{2+k\varepsilon}\left(D,\delta_{D}^{r}\right)$
for $k=0,\ldots,M$.

Assume this is true for $0\leq k<M$. Then by \propref{d-bar-gain-exponent},
\[
A\left(P_{\omega}(f)\wedge\overline{\partial}\varphi\right)\in L^{2+(k+1)\varepsilon}\left(D,\delta_{D}^{r}\right)
\]
and, by \thmref{regularity-P-omega-0}, 
\[
\left(\mathrm{Id}-P_{\omega_{0}}\right)\circ A\left(P_{\omega}(u)\wedge\overline{\partial}\varphi\right)\in L^{2+(k+1)\varepsilon}\left(D,\delta_{D}^{r}\right).
\]
As $\varphi$ is continuous and strictly positive we get $P_{\omega}(f)\in L^{2+(k+1)\varepsilon}\left(D,\delta_{D}^{r}\right)$.

Thus, $P_{\omega}$ maps $L^{p}\left(D,\delta_{D}^{r}\right)$ into
itself for $p\in\left[2,+\infty\right[$. The same result for $p\in\left]1,2\right]$
follows because $P_{\omega}$ is self-adjoint.

\medskip{}

To prove that $P_{\omega}$ maps $L^{p}\left(D,\delta_{D}^{\beta}\right)$
into itself, for $-1<\beta\leq r$, we use a similar induction argument
using \propref{d-bar-gain-weight} instead of \propref{d-bar-gain-exponent}:

For $A$ we choose now the operator $T$ of \propref{d-bar-gain-weight}
with $\gamma=r$, and $\varepsilon$ such that $0<\varepsilon\leq\nicefrac{1}{m}$,
and for which there exists an integer $L$ such that $\beta=r-L\varepsilon$. For $f\in L^{p}\left(D,\delta_{D}^{\beta}\right)$,
assume $P_{\omega}(f)\in L^{2}\left(D,\delta_{D}^{r-l\varepsilon}\right)$,
$0\leq l<L$. Then, \propref{d-bar-gain-weight} and \thmref{regularity-P-omega-0}
imply $\left(\mathrm{Id}-P_{\omega_{0}}\right)\circ A\left(P_{\omega}(u)\wedge\overline{\partial}\varphi\right)\in L^{p}\left(D,\delta_{D}^{r-(l+1)\varepsilon}\right)$
which gives $P_{\omega}(f)\in L^{p}\left(D,\delta_{D}^{r-(l+1)\varepsilon}\right)$.
By induction this gives $P_{\omega}(f)\in L^{p}\left(D,\delta_{D}^{\beta}\right)$,
concluding the proof of (1) of the theorem.

The proof of (2) of the theorem is now easy: assume $u\in\Lambda_{\alpha}(D)$,
$0<\alpha\leq\nicefrac{1}{m}$. Let $p\leq+\infty$ such that $\alpha=\frac{1}{m}\left[1-\frac{m(r+n)+2}{p}\right]$.
By part (1), $P_{\omega}(u)\in L^{p}(D,\delta_{D}^{r})$, by (3) of
\thmref{d-bar-q-gamma'-p-gamma-lip}, $A\left(P_{\omega}(u)\wedge\overline{\partial}\varphi\right)\in\Lambda_{\alpha}(D)$
($A$ being the operator $T$), and, by \thmref{regularity-P-omega-0},
$\left(\mathrm{Id}-P_{\omega_{0}}\right)\circ A\left(P_{\omega}(u)\wedge\overline{\partial}\varphi\right)\in\Lambda_{\alpha}(D)$
concluding the proof.
\begin{rem*}
\quad\mynobreakpar
\begin{enumerate}
\item The restriction $-1<\beta\leq r$ in \thmref{estimates-bergman} (instead
of $0<\beta+1\leq p(r+1)$ in \cite{CDM}) is due to the method because
if $f\in L^{p}\left(D,\delta_{D}^{\beta}\right)$ with $\beta>r$,
a priori $P_{\omega}(f)$ does not exist.
\item The restriction $r\in\mathbb{Q}_{+}$ is not natural and it is very
probable that \thmref{estimates-bergman} is true with $r\in\mathbb{R}_{+}$.
To get that with our method we should first prove the result of \thmref{regularity-P-omega-0}
for $r$ a non negative real number. Looking at the proof in \cite{CDM},
this should be done proving point-wise estimates of the Bergman kernel
of a domain $\widetilde{D}$ of the form
\[
\widetilde{D}=\left\{ (z,w)\in\mathbb{C}^{n+m}\mbox{ such that }\rho_{0}(z)+\sum\left|w_{i}\right|^{2q_{i}}<0\right\} ,
\]
with $q_{i}$ large \emph{real} numbers such that $\sum\nicefrac{1}{q_{i}}=r$.
The difficulty here being that $\widetilde{D}$ is no more $\mathcal{C}^{\infty}$-smooth
and thus the machinery induced by the finite type cannot be used.
\end{enumerate}
\end{rem*}

\bibliographystyle{amsalpha}

\begin{thebibliography}{CDM14b}

\bibitem[AB79]{AB79}
E.~Amar and A.~Bonami, \emph{Mesure de {C}arleson d'ordre $\alpha$ et solutions
  au bord de l'\'equation $\bar\partial$}, Bull. Soc. Math. France \textbf{107}
  (1979), 23--48.

\bibitem[AB82]{BA82}
M.~Andersson and B.~Berndtsson, \emph{Henkin-{R}amirez formulas with weight
  factors}, Ann. Inst. Fourier \textbf{32} (1982), no.~2, 91--110.

\bibitem[AC00]{AC00}
M.~Anderssson and H.~Carlsson, \emph{Estimates of solutions of the ${H}^p$ and
  {BMO} corona problem}, Math. Ann. \textbf{316} (2000), no.~1, 83--102.

\bibitem[Ale01]{Ale01}
W.~Alexandre, \emph{Construction d'une fonction de support \`a la
  {D}iederich-{F}ornaess}, Pub. IRMA, Lille \textbf{54} (2001), no.~III.

\bibitem[Ale11]{Ale11}
\bysame, \emph{A {B}erndtsson-{A}ndersson operator solving
  $\bar\partial$-equation with ${W}^\alpha$-estimates on convex domains of
  finite type}, Math. Z. \textbf{269} (2011), 1155--1180.

\bibitem[Bar92]{Bar92}
D~Barrett, \emph{Behavior of the {B}ergman projection on the
  {D}iederich-{F}orn\ae{}ss worm}, Acta Math. \textbf{168} (1992), no.~1-2,
  1--10.

\bibitem[BCD98]{Bruna-Charp-Dupain-Annals}
J.~Bruna, Ph. Charpentier, and Y.~Dupain, \emph{Zeros varieties for the
  {N}evanlinna class in convex domains of finite type in $\mathbb{C}^n$}, Ann.
  of Math. \textbf{147} (1998), 391--415.

\bibitem[BG95]{BG95}
A.~Bonami and S.~Grellier, \emph{Weighted {B}ergman projections in domains of
  finite type in $\mathbb{C}^2$}, Contemp. Math. \textbf{189} (1995), 65--80.

\bibitem[Cat87]{Catlin-Est.-Sous-ellipt.}
D.~Catlin, \emph{Subelliptic estimates for $\bar{\partial}$-{N}eumann problem
  on pseudoconvex domains}, Annals of Math \textbf{126} (1987), 131--191.

\bibitem[CDM14a]{CDMb}
P.~Charpentier, Y.~Dupain, and M.~Mounkaila, \emph{Estimates for {S}olutions of
  the $\bar\partial$-{E}quation and {A}pplication to the {C}haracterization of
  the {Z}ero {V}arieties of the {F}unctions of the {N}evanlinna {C}lass for
  {L}ineally {C}onvex {D}omains of {F}inite {T}ype}, J. Geom. Anal. \textbf{24}
  (2014), no.~4, 1860--1881.

\bibitem[CDM14b]{CDM}
\bysame, \emph{Estimates for weighted {B}ergman projections on pseudo-convex
  domains of finite type in $\mathbb{C}^n$}, Complex Var. Elliptic Equ.
  \textbf{59} (2014), no.~8, 1070--1095.

\bibitem[CDM15]{CPDY}
\bysame, \emph{On {E}stimates for {W}eighted {B}ergman {P}rojections}, Proc.
  Amer. Math. Soc. \textbf{143} (2015), no.~12, 5337--5352.

\bibitem[Cha80]{Cha80}
P.~Charpentier, \emph{Solutions minimales de l'\'equation $\bar\partial u=f$
  dans la boule et le polydisque}, Ann. Inst. Fourier \textbf{30} (1980),
  121--153.

\bibitem[Chr96]{Christ96}
M.~Christ, \emph{Global $\mathcal{C}^\infty$ irregularity of the
  $\overline\partial$-{N}eumann problem for worm domains}, J. Amer. Math. Soc.
  \textbf{9} (1996), no.~4, 1171--1185.

\bibitem[CKM93]{CKM93}
Z.~Chen, S.~G. Krantz, and D.~Ma, \emph{Optimal ${L}^p$ estimates for the
  $\bar\partial$-equation on complex ellipsoids in $\mathbb{C}^n$}, Manusc.
  Math. \textbf{80} (1993), 131--149.

\bibitem[CL97]{CDC97}
D.~C. Chang and B.~Q. Li, \emph{Sobolev and {L}ipschitz estimates for weighted
  {B}ergman projections}, Nagoya Math. J. \textbf{147} (1997), 147--178.

\bibitem[Con02]{Conrad_lineally_convex}
M.~Conrad, \emph{Anisotrope optimale {P}seudometriken f\"ur lineal konvex
  {G}ebeite von endlichem {T}yp (mit {A}nwendungen)}, PhD thesis,
  Berg.Universit\"at-GHS Wuppertal (2002).

\bibitem[Cum01a]{Cumenge-estimates-holder}
A.~Cumenge, \emph{Sharp estimates for $\bar\partial$ on convex domains of
  finite type}, Ark. Math. \textbf{39} (2001), no.~1, 1--25.

\bibitem[Cum01b]{Cumenge-Navanlinna-convex}
\bysame, \emph{Zero sets of functions in the {N}evanlinna or the
  {N}evanlinna-{D}jrbachian classes}, Pacific J. Math. \textbf{199} (2001),
  no.~1, 79--92.

\bibitem[{\v{C}}Z16]{CZ}
{\v{Z}}.~{\v{C}}u\v{c}kovi\'c and Y.~Zeytuncu, \emph{Mapping {P}roperties of
  {W}eighted {B}ergman {P}rojection {O}perators on {R}einhardt {D}omains},
  Proc. Amer. Math. Soc. \textbf{144} (2016), no.~8, 3479--3491.

\bibitem[DF03]{Diederich-Fornaess-Support-Func-lineally-cvx}
K.~Diederich and J.~E. Fornaess, \emph{Lineally convex domains of finite type:
  holomorphic support functions}, Manuscripta Math. \textbf{112} (2003),
  403--431.

\bibitem[DF06]{Diederich-Fischer_Holder-linally-convex}
K.~Diederich and B.~Fischer, \emph{H\"older estimates on lineally convex
  domains of finite type}, Michigan Math. J. \textbf{54} (2006), no.~2,
  341--452.

\bibitem[DM01]{DM01}
K.~Diederich and E.~Mazzilli, \emph{Zero varieties for the {N}evanlinna class
  on all convex domains of finite type}, Nagoya Math. Journal \textbf{163}
  (2001), 215--227.

\bibitem[Fis01]{MR1815835}
Bert Fischer, \emph{{$L^p$} estimates on convex domains of finite type}, Math.
  Z. \textbf{236} (2001), no.~2, 401--418.

\bibitem[FR75]{FR75}
F.~Forelli and W.~Rudin, \emph{Projections on {S}paces of {H}olomorphic
  {F}unctions in {B}alls}, Indiana Univ. Math. J. \textbf{24} (1975), no.~6,
  593--602.

\bibitem[Hef02]{Hef02}
T.~Hefer, \emph{H\^older and ${L}^p$ estimates for $\bar\partial$ on convex
  domains of finite type depending on {C}atlin's multitype}, Math. Z.
  \textbf{242} (2002), 367--398.

\bibitem[KN65]{Kohn-Nirenberg-1965}
J.~J. Kohn and L.~Nirenberg, \emph{Non coercive boundary value problems}, Comm.
  Pure Appl. Math. \textbf{18} (1965), 443--492.

\bibitem[Koh73]{Kohn-defining-function}
J.~J. Kohn, \emph{Global regularity for $\bar\partial$ on weakly pseudo-convex
  manifolds}, Trans. Amer. Math. Soc. \textbf{181} (1973), 273--292.

\bibitem[Lig89]{Lig89}
E.~Ligocka, \emph{On the {F}orelli-{R}udin construction and weighted {B}ergman
  projections}, Studia Math. \textbf{94} (1989), no.~3, 257--272.

\bibitem[Ngu01]{Ngu01}
N.~Nguyen, \emph{Un th\'eor\`eme de la couronne ${H}^p$ et z\'eros des
  fonctions de ${H}^p$ dans les convexes de type fini}, Pr\'epublication
  n\mbox{$^\circ$} 224, Laboratoire Emile Picard, Universit\'e de Toulouse III
  (2001).

\bibitem[{Ran}86]{RM86}
{Range M.}, \emph{Holomorphic {F}unctions and {I}ntegrals {R}epresentations in
  {S}everal {C}omplex {V}ariables}, Springer-Verlag, 1986.

\bibitem[Zey11]{Zey11}
Y.~Zeytuncu, \emph{Weighted {B}ergman projections and kernels: ${L}^p$
  regularity and zeros}, Proc. Amer. Math. Soc. \textbf{139} (2011), no.~6,
  2105--2112.

\bibitem[Zey12]{Zey12}
\bysame, \emph{${L}^p$ regularity of some weighted {B}ergman projections on the
  unit disc}, Turkish J. Math. \textbf{36} (2012), no.~3, 386--394.

\bibitem[Zey13a]{Zey13b}
\bysame, \emph{${L}^p$ {R}egularity of weighted {B}ergman {P}rojections}, Tran.
  Amer. Math. Soc. \textbf{365} (2013), no.~6, 2959--2976.

\bibitem[Zey13b]{Zey13a}
\bysame, \emph{Sobolev regularity of weighted {B}ergman projections on the unit
  disc}, Complex Var. Elliptic Equ. \textbf{58} (2013), no.~3, 309--315.

\bibitem[Zey16]{Zey}
\bysame, \emph{An {A}pplication of the {P}r\'ekopa-{L}eindler {I}nequality and
  {S}obolev {R}egularity of {W}eighted {B}ergman {P}rojections}, To appear in
  Acta Sci. Math. (Szeged) (2016).

\end{thebibliography}

\providecommand{\bysame}{\leavevmode\hbox to3em{\hrulefill}\thinspace}
\providecommand{\MR}{\relax\ifhmode\unskip\space\fi MR }
\providecommand{\MRhref}[2]{%
  \href{http://www.ams.org/mathscinet-getitem?mr=#1}{#2}
}
\providecommand{\href}[2]{#2}

\end{document}